\documentclass[a4paper, 11pt, reqno]{amsart}
\allowdisplaybreaks

\usepackage[margin=2.15cm]{geometry}
\usepackage{graphicx}
\usepackage{amsmath}
\usepackage{amsthm,amsfonts,amssymb,mathrsfs,amscd,amstext,amsbsy}
\usepackage{epic,eepic}
\usepackage{enumerate}
\usepackage[all]{xy}
\usepackage{tikz}
\usepackage{color}
\usepackage{verbatim}

\usetikzlibrary{patterns}
\usetikzlibrary{decorations.markings}
\usetikzlibrary{arrows}

\usepackage[english]{babel}
\usepackage{enumitem}
\usepackage{hyperref}

\hypersetup{
    pdftitle=   {On 1-factors with prescribed lengths in tournaments},
   pdfauthor=  {Dong Yeap Kang, Jaehoon Kim}
}

\newtheorem{THM}{Theorem}[section]
\newtheorem*{THM*}{Theorem~\ref{main}}
\newtheorem{LEM}[THM]{Lemma}

\newtheorem{FACT}[THM]{Fact}
\newtheorem{COR}[THM]{Corollary}
\newtheorem{PROP}[THM]{Proposition}

\newtheorem{CLAIM}{Claim}
\newtheorem{QUE}{Question}
\theoremstyle{remark}

\newcommand\Int{\mathrm{Int}}

\newcommand{\COMMENT}[1]{}

\tikzset{middledownarrow/.style={
        decoration={markings,
            mark= at position 0.4 with {\arrow{#1}} ,
        },
        postaction={decorate}
    }
}

\tikzset{middlearrow/.style={
        decoration={markings,
            mark= at position 0.6 with {\arrow{#1}} ,
        },
        postaction={decorate}
    }
}

\tikzset{middleuparrow/.style={
        decoration={markings,
            mark= at position 0.75 with {\arrow{#1}} ,
        },
        postaction={decorate}
    }
}

\tikzset{middleupuparrow/.style={
        decoration={markings,
            mark= at position 0.9 with {\arrow{#1}} ,
        },
        postaction={decorate}
    }
}

\tikzset{middleupupuparrow/.style={
        decoration={markings,
            mark= at position 0.95 with {\arrow{#1}} ,
        },
        postaction={decorate}
    }
}

\tikzset{middleupupupuparrow/.style={
        decoration={markings,
            mark= at position 1 with {\arrow{#1}} ,
        },
        postaction={decorate}
    }
}
\theoremstyle{definition}

\numberwithin{equation}{section}

\title{On $1$-factors with prescribed lengths in tournaments}

\author{Dong Yeap Kang}
\address[Dong Yeap Kang]{
Department of Mathematical Sciences, KAIST, Daejeon 34141, South Korea, and Discrete Mathematics Group, Institute for Basic Science (IBS), Daejeon 34126, South Korea}
\email[Dong Yeap Kang]{dyk90@kaist.ac.kr}

\author{Jaehoon Kim}
\address[Jaehoon Kim]{Mathematics Institute, University of Warwick, Coventry, United Kingdom CV4 7AL}
\email[Jaehoon Kim]{Jaehoon.Kim.1@warwick.ac.uk, mutualteon@gmail.com}

\thanks{The first author has been supported by the National Research Foundation of Korea (NRF) grant funded 
by the Korea government (MSIT) (No. NRF-2017R1A2B4005020), by IBS-R029-C1, and also by TJ Park Science Fellowship of POSCO 
TJ Park Foundation.
The second author was partially supported by the European Research Council under the European Union's Seventh Framework Programme (FP/2007--2013) / ERC Grant Agreements no. 306349 and also partially supported by the Leverhulme Trust Early Career
Fellowship ECF-2018-538 (J. Kim).
}

\date{\today}

\begin{document}

\begin{abstract}
We prove that every strongly $10^{50}t$-connected tournament contains all possible $1$-factors with at most $t$ components and this is best possible up to constant.
 In addition, we can ensure that each cycle in the $1$-factor contains a prescribed vertex. This answers a question by K\"uhn, Osthus, and Townsend.

Indeed, we prove more results on partitioning tournaments. We prove that a strongly $\Omega(k^4tq)$-connected tournament admits a vertex partition into $t$ strongly $k$-connected tournaments with prescribed sizes such that each tournament contains $q$ prescribed vertices, provided that the prescribed sizes are $\Omega(n)$. This result improves the earlier result of K\"uhn, Osthus, and Townsend. We also prove that for a strongly $\Omega(t)$-connected $n$-vertex tournament $T$ and given $2t$ distinct vertices $x_1,\dots, x_{t}, y_{1},\dots, y_{t}$ of $T$, we can find $t$ vertex disjoint paths $P_1,\dots, P_t$ such that each path $P_i$ connecting $x_i$ and $y_i$ has the prescribed length, provided that the prescribed lengths are $\Omega(n)$. For both results, the condition of  connectivity being linear in $t$ is best possible, and the condition of prescribed sizes being $\Omega(n)$ is also best possible.
\end{abstract}

\maketitle

\section{Introduction}\label{sec:intro}

\subsection{Disjoint cycles in tournaments with prescribed lengths and vertices}
A \emph{$1$-factor} in a digraph $D$ is a spanning subgraph that is a union of vertex-disjoint cycles of length at least three in $D$. 

In 1959, Camion~\cite{camion1959} proved that a tournament admits a Hamiltonian cycle (a cycle containing all vertices) if and only if it is strongly 
connected. As a generalization of Camion's theorem, Bollob\'as (see~\cite{reid1989}) posed the following question: for each $t\in \mathbb{N}$, what is the least integer $g(t)$ such that all strongly $g(t)$-connected tournaments, up to finitely many exceptions, contain a  $1$-factor with exactly $t$ components?  Clearly $g(1)=1$ by Camion's theorem, and it is easy to see that $g(t)$ exists and $g(t) \geq t$ for each $t \in \mathbb{N}$.
Reid~\cite{reid1985} proved $g(2)=2$ by showing that every strongly 2-connected $n$-vertex tournament $T$ with $n \geq 6$ contains two vertex-disjoint cycles $C_1$ and $C_2$ with $|V(C_1)|=3$ and $|V(C_2)|=n-3$, if $T$ is not isomorphic to 
the $7$-vertex tournament with no transitive $4$-vertex subtournament. Finally, Chen, Gould, and Li~\cite{chen2001} resolved the question of Bollob\'as by proving that every strongly $t$-connected $n$-vertex tournament with $n \geq 8t$ contains $t$ vertex-disjoint cycles covering all vertices of the tournament, implying $g(t) = t$. 

As these results only guarantee the existence of a $1$-factor with $t$ components in highly connected tournaments, it is natural to ask when tournaments have all possible $1$-factors with $t$ components (i.e. contain pairwise vertex-disjoint cycles of lengths $\ell_1,\dots, \ell_t$ for all tuples $(\ell_1,\dots, \ell_t)$ of natural numbers with $\sum_{i=1}^{t}\ell_i = n$ and $\ell_i\geq 3$).
Song~\cite{song1993} extended the result of Reid by proving that every strongly $2$-connected $n$-vertex tournament with $n\geq 6$ contains all possible $1$-factors with two components as long as $T$ is not isomorphic to the $7$-vertex tournament with no transitive $4$-vertex subtournament, and posed the following question analogous to the question of Bollob\'as: for each $t\in \mathbb{N}$, what is the least integer $f(t)$ such that all strongly $f(t)$-connected tournaments, up to finitely many exceptions, contain all possible $1$-factors with exactly $t$ components? Clearly, the result of Song~\cite{song1993} shows $f(2)=2$, and it was conjectured that $f(t)=g(t)$ for all $t \in \mathbb{N}$.

Recently, K\"uhn, Osthus, and Townsend~\cite{kuhn2016cycle} extended the result of Song by proving that every strongly $10^{10} t^4 \log t$-connected $n$-vertex tournament contains all possible $1$-factors with at most $t$ components. More precisely, they showed that, for any integers $\ell_1,\dots, \ell_t \geq 3$ with $\sum_{i\in [t]} \ell_i=n$, the tournament $T$ contains $t$ vertex-disjoint cycles $C_1 , \dots , C_t$ such that $|V(C_i)| = \ell_i$ for each $i\in [t]$. They asked 
whether the connectivity $10^{10} t^4 \log t$ could be reduced to $O(t)$. Later in \cite[Problem 5.2]{pokrovskiy2016edge}, Pokrovskiy asked the same question again. In Theorem~\ref{thm:cycle}, we answer their question in the affirmative.

On the other hand, Moon~\cite{moon1966} proved another generalization of Camion's theorem stating that every strongly connected tournmanet $T$ is \emph{vertex-pancyclic}, meaning that for every vertex $v\in V(T)$ and any integer $3\leq \ell\leq |V(T)|$, there exists a cycle of length $\ell$ containing $v$ in $T$. Bang-Jensen, Guo and Yeo~\cite{jensen2000} proved that for any strongly $3$-connected tournament $T$ with at least $8$ vertices and two distinct vertices $v_1,v_2\in V(T)$, the tournament $T$ can be partitioned into two vertex-disjoint cycles $C_1$ containing $v_1$ and $C_2$ containing $v_2$.
In Theorem~\ref{thm:cycle} we proved that one can guarantee a much stronger pancyclicity with $t$ cycles if the connectivity of the tournament is linear in $t$. 

\begin{THM}\label{thm:cycle}
Let $n,t, \ell_1,\dots, \ell_t \in \mathbb{N}$ with $\ell_1,\dots, \ell_t\geq 3$ and $\sum_{i\in [t]} \ell_i = n$. For any strongly $10^{50} t$-connected $n$-vertex tournament $T$ and $t$ distinct vertices $x_1 , \dots , x_t \in V(T)$, the tournament $T$ contains vertex-disjoint cycles $C_1 , \dots , C_t$ such that $x_i \in V(C_i)$ and $|V(C_i)| = \ell_i$ for each $i\in [t]$.
\end{THM}

The connectivity bound is sharp up to a multiplicative constant, and we do not attempt to optimize the constant $10^{50}$. 
In particular, Theorem~\ref{thm:cycle} implies that $t \leq f(t) \leq 10^{50}t$. It would be interesting if one can prove $f(t) = t$, answering the conjecture of Song~\cite{song1993}.  Note that one can easily extend Theorem~\ref{thm:cycle} to strongly $10^{51}t$-connected semicomplete digraphs, as every strongly $(3k-2)$-connected semicomplete digraph contains a strongly $k$-connected spanning tournament~\cite{guo1997}. 

There are some related results in different settings. Amar and Raspaud~\cite{amar1991} proved that every strongly connected $n$-vertex digraph with at least $(n-1)(n-2)+3$ edges contains all possible $1$-factors except in two cases. 
Keevash and Sudakov \cite{keevash2009cycle} proved that every $n$-vertex oriented graph with minimum semidegree close to $n/2$ admits $t$ vertex-disjoint cycles with prescribed lengths covering almost all vertices. For undirected graphs, the El-Zahar conjecture determines the minimum degree condition guaranteeing a partion of an $n$-vertex graph into vertex-disjoint cycles of prescribed lengths, and was proved for all large $n$ by Abbasi~\cite{Abbasi}. For more on topics and results related to $1$-factors in digraphs, the readers are referred to~\cite[Chapter 13]{bang2008digraphs}.

We end this subsection by noting that Theorem~\ref{thm:cycle} is best possible in the following sense. 
\begin{itemize}
\item
A partition of a highly connected tournament $T$ into strongly $k$-connected subgraphs of prescribed sizes for $k\geq 2$ may not exist (observe that Theorem~\ref{thm:cycle} is the case of $k=1$).

\item A partition of a highly connected tournament $T$ into cycles of prescribed lengths containing at least two prescribed vertices may not exist.
\end{itemize}
Indeed, the following proposition presents a highly connected tournament $T$ with diameter at least $\Omega(n)$ such that every strongly $k$-connected subgraph of $T$ with $k\geq 2$ contains at least $\Omega(n)$ vertices. Note that the diameter of $T$ implies that there are two vertices $x,y$ with distance $\Omega(n)$, thus any cycle containing both $x$ and $y$ must have length at least $\Omega(n)$.

\begin{PROP}\label{prop:example}
For $k, s, n \in \mathbb{N}$ with $2\leq k\leq s$ and $2\binom{s+1}{2}+2s+2\leq n$, there exists an $n$-vertex strongly $s$-connected tournament $T$ of diameter at least $\binom{s+1}{2}^{-1} (n - 2s)$ such that every strongly $k$-connected subtournament $T'$ of $T$ satisfies $|V(T')|\geq k\binom{s+1}{2}^{-1}n-k-2$.
\end{PROP}

 However, if we further assume that all the prescribed sizes are $\Omega(n)$, then both generalizations of Theorem~\ref{thm:cycle} become true. Theorem~\ref{thm:balanced} shows that such a partition exists, provided that all prescribed sizes are $\Omega(n)$.

\subsection{Partitioning tournaments into highly connected subtournaments with prescribed sizes}
Thomassen~\cite{reid1989} asked whether for integers $k_1 , \dots , k_t$, there exists $f(k_1 , \dots, k_t)$ such that every strongly $f(k_1 , \dots , k_t)$-connected tournament admits a vertex-partition $
W_1 , \dots , W_t$ such that for each $i\in [t]$, the set $W_i$ induces a strongly $k_i$-connected subtournament. 

The following two concepts were introduced to tackle the above problem.
A tournament $T$ is \emph{critically strongly $k$-connected} if $T$ is strongly $k$-connected but it is no longer strongly $k$-connected after deleting any vertex. A tournament is \emph{minimally strongly $k$-connected tournament} if it is strongly $k$-connected but any proper subtournament is not strongly $k$-connected. 

By Moon's theorem, the cycle of length three is the only critically (or minimally) strongly $1$-connected tournament. However, Thomassen (see~\cite[Theorem 2.14.11]{bang2018classes}) showed that there are infinitely many critically strongly $k$-connected tournaments for every integer $k \geq 2$. 
On the other hand, Lichiardopol~\cite{lichiardopol2012} (see also~\cite[Conjecture 2.14.12]{bang2018classes}) conjectured that the situation is different for minimally strongly $k$-connected tournaments: for every integer $k \geq 1$, there are only finitely many minimally strongly $k$-connected tournaments, which yields the positive answer to the Thomassen's question if it is true. However, the following corollary disproves Lichiardopol's conjecture.
One can easily prove this corollary by considering a smallest strongly $k$-connected subtournament of large strongly $k$-connected tournament $T$ in Proposition~\ref{prop:example}. 

\begin{COR}\label{cor:disproof}
For every integer $k \geq 2$, there are infinitely many minimally strongly $k$-connected tournaments.
\end{COR}

Nonetheless, Thomassen's problem turned out to be true: K\"uhn, Osthus, and Townsend~\cite{kuhn2016cycle} recently answered this question in the affirmative, and proved that every strongly $10^7 k^6 t^3 \log (kt^2)$-connected 
tournament can be partitioned into $t$ strongly $k$-connected subtournaments. Moreover, they also proved that the sizes of $W_i$ can be prescribed  as long as the prescribed sizes are not too small. We improve their result in the following theorem. 
Here, $\delta(D)$ denotes $\min_{v\in V(D)} \big\{|N^+_D(v)\cup N^-_D(v)|\big\}$.

\begin{THM}\label{thm:balanced}
Let $k, t, \ell, m,n,q, a_1,\dots, a_t \in \mathbb{N}$ with $t,m\geq 2$, $\sum_{i\in [t]} a_i \leq n$ and $a_i\geq n/(10tm)$ for each $i\in [t]$.
Suppose that $D$ is a strongly $10^{8} qk^2 \ell (k+\ell)^2  tm^2 \log(m) $-connected $n$-vertex digraph with $\delta(D)\geq n-\ell$, and $Q_1 , \dots , Q_t \subseteq V(D)$ are $t$ disjoint sets with $|Q_i| \leq q$ for each $i \in [t]$.
Then there exists a partition $W_1 , \dots , W_t$ of $V(D)$ satisfying the following.
\begin{itemize}[noitemsep]
\item[$(\rm 1)$] $Q_i \subseteq W_i$.
\item[$(\rm 2)$] For every $i\in [t]$, the subgraph $D[W_i]$ is strongly $k$-connected.
\item[$(\rm 3)$] $|W_i| = a_i$.
\end{itemize}
\end{THM}

Note that the condition $\delta(D) \geq n - \ell$ in the theorem ensures that the digraph $D$ is close to being semicomplete. Compared to the result in \cite{kuhn2016cycle}, we improve the connectivity bound to $O(t)$ that is best possible. Note that there are infinitely many strongly $(kt-1)$-connected tournaments without $t$ pairwise vertex-disjoint strongly $k$-connected subtournaments.

Also, we extend the theorem to digraphs which are close to semicomplete, rather than just for tournaments, and to allow each part to contain a small set of prescribed vertices, where the lower bound $\Omega(n)$ on the size $a_i$ is necessary by Proposition~\ref{prop:example}.

The dependence on $k,m,\ell$ in the connectivity bound is unlikely to be best possible. Indeed, K\"uhn, Osthus and Townsend~\cite{kuhn2016cycle} conjectured that connectivity $O(kt)$ suffices for tournaments when we do not consider the prescribed sizes.
For the clarity of statement and the argument, we have not attempted to improve the orders of $m$ and $\ell$ in the statements. In fact, it is very easy to make small improvements of the orders of $m$ and $\ell$.\footnote{
If we assume that $n$ is large, then the connectivity can be easily replaced with $10^{8} qk^2 \ell (k+\ell)^2  tm \log(m)$. Note that the only place we need the term $m^2\log(m)$ instead of $m\log(m)$ is \eqref{eq: W'i size}. However, as long as $n$ is large, we can still obtain \eqref{eq: W'i size} with the smaller connectivity.
Also, by sacrificing the order of $k$, one can easily improve the order of $\ell$.
Note that \ref{U*1} in Claim~\ref{cl: extension} can be easily replaced by $|U^*_i|\leq \min\{4k(k+\ell), (k+1)^2|U_i|\}$ assuming $|U_i| =O(\ell^2 k^2tm\log(km\ell))$. By altering some parameters and calculations in the proof, we can show that the connectivity $O(qk^5 \ell^2 m^2 t \log(km\ell))$ suffices.
}

We remark that Theorem~\ref{thm:balanced} proves a generalization of the following conjecture of Bang-Jensen, Guo and Yeo~\cite{jensen2000}, which states that there is a function $h(k_1,k_2)$ such that for all strongly $h(k_1,k_2)$-connected tournaments $T$ and $v_1,v_2 \in V(T)$, $T$ can be partitioned into vertex-disjoint tournaments $T_1$ containing $v_1$ and $T_2$ containing $v_2$ such that each $T_i$ is strongly $k_i$-connected. 
\footnote{Indeed, a few modification of the proof of the theorem by K\"uhn, Osthus and Townsend~\cite{kuhn2016cycle} already implies the existence of $h(k_1,k_2)$.} 

Now we end this subsection with some related open questions.  Hajnal~\cite{hajnal1983} and Thomassen~\cite{thomassen1983} proved an analogous theorem  for undirected graphs. They proved that there exists $f(k)$ such that any $f(k)$-connected graph can be partitioned into two $k$-connected subgraphs. Later, K\"uhn and Osthus~\cite{kuhn2003} further generalized this. It would be very interesting if one can prove Theorem~\ref{thm:balanced} with a connectivity independent of $\ell$. Indeed, K\"uhn, Osthus, and Townsend~\cite{kuhn2016cycle} asked the following question.
\begin{QUE}\cite{kuhn2016cycle}\label{que:gen_balanced}
For all $k,t \in \mathbb{N}$, does there exist $f(k,t)$ such that every strongly $f(k,t)$-connected digraph $D$ can be partitioned into $t$ vertex-disjoint strongly $k$-connected subdigraphs?
\end{QUE}

Stiebitz~\cite{stiebitz96} considered similar question for minimum degree instead of connectivity and proved that for integers $s,t \geq 0$ and an undirected graph $G$ with minimum degree at least $s+t+1$, there exist two disjoint sets $A,B \subseteq V(G)$ with $V(G)=A\cup B$ such that both $G[A]$ and $G[B]$ have minimum degree at least $s$ and $t$, respectively. He further asked the following question, whether the analogue of this result also holds for digraphs.

\begin{QUE}\cite{stiebitz95}\label{que:gen_degree}
For integers $s,t \geq 0$, does there exist $g(s,t)$ such that every digraph $D$ with minimum out-degree at least $g(s,t)$ can be partitioned into two vertex-disjoint subdigraphs $D_1$ and $D_2$ with minimum out-degree $s$ and $t$ respectively?
\end{QUE}

\subsection{Disjoint paths in tournaments with prescribed lengths connecting prescribed vertices}

For an integer $t \geq 1$, a digraph $D$ is \emph{$t$-linked} if $|V(D)| \geq 2t$ and for any $2t$ distinct\footnote{In fact, one may remove this condition by replacing some vertices with their neighbors.} vertices $x_1 , \dots , x_t , y_1 , \dots , y_t$ of $D$, there exist $t$ vertex-disjoint paths $P_1,\dots, P_t$ such that for each $i\in [t]$ the path $P_i$ starts at $x_i$ and ends at $y_i$.  Note that it is easy to see  that every $t$-linked digraph is strongly $(2t-1)$-connected, 
and it is natural to ask whether highly connected digraphs are highly linked. For undirected graphs, Bollob\'as and Thomason~\cite{bollobas1996highly} proved that every $22t$-connected 
graph is $t$-linked. However, Thomassen~\cite{thomassen1991} proved that for every integer $m \geq 1$, 
there exists a strongly $m$-connected digraph that is not $2$-linked, so high connectivity does not imply high linkedness for digraphs in general.

On the other hand, Thomassen~\cite{thomassen1984} showed that there exists a constant $C$ such that any strongly $2^{C t \log t}$-connected tournament is $t$-linked, and thus high connectivity implies high linkedness for tournaments. Recently, K\"uhn, Lapinskas, Osthus and Patel~\cite{kuhn2014proof} improved the connectivity bound to $O(t \log t)$, and finally  Pokrovskiy~\cite{pokrovskiy2015highly} gave a linear bound on the connectivity that any strongly $452t$-connected tournament is $t$-linked. Gir\~ao and Snyder \cite{GS18tournamentlinkage} proved that every strongly $4t$-connected tournament $T$ is $t$-linked if the minimum out-degree of $T$ is large. Pokrovskiy~\cite{pokrovskiy2016edge} also proved that there exists a constant $C$ such that every strongly $Ct$-connected tournament
is $t$-linked, so that the union of $t$ paths connecting given $t$ pairs of vertices covers all vertices. 

\begin{THM}[Pokrovskiy~\cite{pokrovskiy2016edge}]\label{thm:spanlink}
There exists $c>0$ such that for any integer $t \geq 1$, every strongly $ct$-connected $n$-vertex tournament $T$ with $2t$ distinct vertices $x_1 , \dots , x_t , y_1 , \dots , y_t \in V(T)$ admits $t$ vertex-disjoint paths $P_1 , \dots , P_t$ so that $P_i$ is a path from $x_i$ to $y_i$ for $1 \leq i \leq t$ and $\sum_{i=1}^{t} |V(P_i)| = n$.
\end{THM} 

By using Theorem~\ref{thm:balanced}, we can easily derive Theorem~\ref{thm:controlpath} which provides us vertex disjoint paths of prescribed lengths connecting prescribed vertices.

\begin{THM}\label{thm:controlpath}
For integers $n,t \geq 1$ and $\ell_1 , \dots , \ell_t \geq n/(100t)$ with $\sum_{i=1}^{t}\ell_i \leq n$, every strongly $10^{14} t$-connected semicomplete digraph $D$ with $2t$ distinct vertices $x_1,\dots,x_t,y_1,\dots,y_t \in V(D)$ contains $t$ vertex-disjoint paths $P_1 , \dots , P_t$ so that $P_i$  starts at $x_i$ and ends at $y_i$ and $|V(P_i)| =\ell_i$ for each $i\in [t]$.
\end{THM}
\begin{proof}
Thomassen~\cite{thomassen1980} proved that given any strongly 4-connected semicomplete digraph $D$ and two vertices $x \ne y$ of $D$, there is a Hamiltonian path from $x$ to $y$. Applying Theorem~\ref{thm:balanced} with $(k,\ell,m,q,a_i)=(4,1,10,2, \ell_i)$ and using the result of Thomassen, the theorem follows.
\end{proof}

By Proposition~\ref{prop:example}, the lower bound $\Omega(n)$ on the length of paths is best possible in the sense that the digraph $D$ may have the diameter linear in $n$.

We will derive both Theorems~\ref{thm:cycle} and~\ref{thm:balanced} from Lemma~\ref{lem:main} that provides a powerful connectivity structures of tournament-like digraphs, which may be of independent interest. To prove Lemma~\ref{lem:main}, we use the concept of robust linkage structures introduced by K\"uhn, Lapinskas, Osthus and Patel in \cite{kuhn2014proof}. Robust linkage structure is a very useful tool providing ``skeletons'' of highly connected tournaments that gives control on the connectivity. Further results were obtained by using this method \cite{kangstrong2018, kang2017sparse, kim2016bipartitions, kuhn2016cycle, pokrovskiy2015highly, pokrovskiy2016edge}. We also remark that  the sparse linkage structure introduced in \cite{kang2017sparse} is useful for our proof.

\section{Preliminaries and Tools}\label{sec:prelim}

\subsection{Basic terminology}\label{sec:basic}
Let $\mathbb{N}=\{1,2,\dots\}$ and for an integer $n$, we denote $[n] :=  \{1,\dots, n \}$. In particular, $[n] = \emptyset$ if $n \leq 0$. 
We 
always denote logarithm as $\log:= \log_2$.
For $k\in \mathbb{N}$ and tuples $(i_1,\dots, i_k), (j_1,\dots, j_k) \in \mathbb{N}^k$, we write $(i_1,\dots, i_k)< (j_1,\dots, j_k)$ 
if $i_s < j_s$ where $s= \min\{s'\in [k]: i_{s'}\neq j_{s'}  \}.$
We omit floors and ceilings and treat all large numbers as integers whenever it does not affect our argument.

We denote $D= (V(D), E(D))$ as a \emph{directed graph} or \emph{digraph} if $V(D)$ is a finite set and $E(D) \subseteq \{ \overrightarrow{uv} : u\neq v \in V(D)\}$.
For a vertex $v$ of a digraph $D$, we let
$N_D^-(v):= \{ u\in V(D): \overrightarrow{uv} \in E(D)\}$, $N_D^{+}(v):=\{ u\in V(D): \overrightarrow{vu} \in E(D)\}$ and $
N_D(v) := N_D^{-}(v) \cup N_D^+(v)$
be the set of \emph{in-neighbors}, the set of \emph{out-neighbors} and the set of \emph{neighbors} of $v$, respectively.
For $v\in V(D)$, let 
\begin{align*}
d_D^-(v)&:=|N_D^-(v)|, &d_D^+(v)&:=|N_D^+(v)|,& 
d_D(v)&:= |N_D(v)|,\\
\delta^-(D)&:=\min_{v\in V(D)} \{d_D^{-}(v)\},
&\delta^+(D)&:=\min_{v\in V(D)} \{d_D^{+}(v)\} \enspace
\text{and} &\delta(D)&:=\min_{v\in V(D)} \{d(v)\}.
\end{align*}

A digraph $D$ is \emph{semicomplete} if $\delta(D)=n-1$, and a semicomplete digraph $D$ is a \emph
{tournament} if it does not contain a cycle of length two. For a path $P$ of $D$, we write $P = (v_1 , 
\dots , v_k)$ if $P$ is a path with $V(P) = \left \{v_1 , \dots , v_k \right \}$ and $E(P) = \left \{
\overrightarrow{v_iv_{i+1}} \: : \: i \in [k-1] \right \}$ and we write $\Int(P):=\{ v_2,\dots, v_{k-1}\}$.
We say a cycle $C$ is a \emph{Hamiltonian cycle} of $D$ if 
$V(C)=V(D)$.
For two digraphs $D$ and $D'$, we say that $D'$ is a \emph{subgraph} of $D$ if $V(D')\subseteq V(D)$ and $E
(D')\subseteq E(D)$.
For a given set $U\subseteq V(D)$, we write $D[U]$ to denote the digraph with vertex set $U$ and edge set $\{\overrightarrow{uv}: 
u,v\in U, \overrightarrow{uv} \in E(D)\}$ and
we write $D\setminus U:= D[V(D)\setminus U]$.
For disjoint sets $U,V \subseteq V(D)$, we write $D[U,V]$ to denote the digraph with vertex set $U\cup V$ and edge set $\{ \overrightarrow{uv} \in E(D): |\{u,v\}\cap U|=|\{u,v\}\cap V| =1\}$.

For $k\in \mathbb{N}$, we say that an ordered pair $(u,v) \in V(D)\times V(D)$ is \emph{$k$-connected in $D$} if for any subset $S\subseteq V(D)\setminus\{u,v\}$ with $|S|\leq k-1$, there exists a path from $u$ to $v$ in $D\setminus S$.
We say that a digraph $D$ is \emph{strongly $k$-connected} if $|V(D)|\geq k+1$ and every ordered pair $(u,v) \in V(D)\times V(D)$ is $k$-connected in $D$.
For a vertex $v\in V(D)$ and a set $U$, we say that 
$(v,U)$ is \emph{$k$-connected in $D$} if for any subset $S\subseteq V(D)\setminus\{v\}$ with $|S| \leq k-1$, there exists a path from $v$ to a vertex in $U\setminus S$ in $D\setminus S$.
Similarly, we say $(U,v)$ is \emph{$k$-connected in $D$} if for any subset $S\subseteq V(D)\setminus\{v\}$ with $|S| \leq k-1$, there exists a path from a vertex in $U\setminus S$ to $v$ in $D\setminus S$.
Note that $U$ does not have to be a subset of $V(D)$ in this definition.
However, if $(v,U)$ is $k$-connected in $D$ or $(U,v)$ is $k$-connected in $D$, then either $v\in U$ or $|U\cap V(D)| \geq k$. Note that for $k\geq 2$ and $u\neq v \in V(D)$, the pair $(u,\{v\})$ is never $k$-connected in $D$ while $(u,v)$ may be $k$-connected in $D$.


\subsection{Some lemmas}
Now we state some basic results we use later in the proof. The following can be easily deduced by using Hall's theorem. We omit the proof.
\begin{FACT}\label{fact: perfect matching}
Let $G$ be a bipartite graph on vertex partition $(A,B)$ with $|A|=|B|$.
If  $d_G(a)+d_G(b)\geq |A|$ for any $(a,b)\in A\times B$, then $G$ has a perfect matching.
\end{FACT}
The following theorem by Camion~\cite{camion1959} is useful to find a cycle of certain length in a tournament.
\begin{THM}\cite{camion1959}\label{thm:camion}
Every strongly connected tournament contains a Hamiltonian cycle.
\end{THM}

Moon~\cite{moon1966} extended the result of Camion~\cite{camion1959} by proving the following Theorem.
\begin{THM}\cite{moon1966}\label{thm:moon}
Let $k$ and $n$ be integers with $3 \leq k \leq n$. For any strongly connected $n$-vertex tournament $T$ and a vertex $v \in V(T)$, $T$ contains a cycle $C$ with $|V(C)|=k$ and $v \in V(C)$.
\end{THM}

The following lemma can be proved by using basic definition of strong $k$-connectivity. We omit the proof.

\begin{LEM}\label{lem:glue}
Let $k\in \mathbb{N}$ and $D, D_1, D_2$ be directed graphs and $U,W$ be vertex sets. Then the following hold.
\begin{itemize}
\item For $v\in V(D_1\cup D_2)$, if $(v,U)$ is $k$-connected in $D_1$ and $(u,W)$ is $k$-connected in $D_2$ for all $u\in U$, then $(v,W)$ is $k$-connected in $D_1\cup D_2$.
 \item For $v\in V(D_1\cup D_2)$, if $(U,v)$ is $k$-connected in $D_1$ and $(W,u)$ is $k$-connected in $D_2$ for all $u\in U$, then $(W,v)$ is $k$-connected in $D_1\cup D_2$.
 \item Suppose $D[W]$ is a strongly $k$-connected digraph. If $(u,W)$ and $(W,u)$ are both $k$-connected in $D[U\cup W]$ for all $u\in U$, then $D[U\cup W]$ is strongly $k$-connected digraph.
 \end{itemize}
\end{LEM}

The following theorem by Kim, K\"uhn and Osthus~\cite{kim2016bipartitions} is useful to prove Corollary~\ref{cor:twocycles}.
\begin{THM}\cite{kim2016bipartitions}\label{thm:kko}
For an integer $k \geq 1$ and a strongly $10^9 k^6 \log (2k)$-connected tournament $T$, there exists a partition $V_1,V_2$ of $V(T)$ such that 
$T[V_1]$, $T[V_2]$ and $T[V_1,V_2]$ are strongly $k$-connected.
\end{THM}

The following corollary will be useful to prove Theorem~\ref{thm:cycle}. Indeed, slight modification of the proof of Theorem~\ref{thm:kko} gives us a more general result, but Corollary~\ref{cor:twocycles} is sufficient for our purpose.

\begin{COR}\label{cor:twocycles}
Let $\ell$ and $n$ be integers with $n\geq 6$ and $3\leq \ell \leq n-3$ and $T$ be a strongly $10^9$-connected $n$-vertex tournament with a vertex $v \in V(T)$. There exist two vertex-disjoint cycles $C_1$ and $C_2$ in $T$ such that $v \in V(C_1)$, $|V(C_1)| =\ell$ and $|V(C_2)| = n-\ell$.
\end{COR}
\begin{proof}
By Theorem~\ref{thm:kko}, there exists a partition $V_1,V_2$ of $V(T)$ such that $T[V_1]$, $T[V_2]$, and $T[V_1,V_2]$ are strongly connected. We may assume that $v \in V_1$. 

If $\ell \leq |V_1|$, then by Theorem~\ref{thm:moon}, $T[V_1]$ contains a cycle $C_1$ of length $\ell$ that contains $v$. 
Since $T[V_1,V_2]$ is strongly connected, every vertex $v\in V_1$ has an out-neighbor and an in-neighbor in $V_2$.
As $T[V_2]$ is strongly connected, Lemma~\ref{lem:glue} implies that $T[(V_1 \setminus V(C_1)) \cup V_2]$ is strongly connected, thus Theorem~\ref{thm:camion} implies that $T[(V_1 \setminus V(C_1)) \cup V_2]$ contains a spanning cycle $C_2$ of length $n-\ell$.

If $\ell > |V_1|$, then by Theorem~\ref{thm:moon}, $T[V_2]$ contains a cycle $C_2$ of length $n-\ell$. Again, since $T[V_1,V_2]$ and $T[V_1]$ are strongly connected, Lemma~\ref{lem:glue} implies that $T[V_1 \cup (V_2 \setminus V(C_2))]$ is strongly connected, thus Theorem~\ref{thm:camion} implies that $T[V_1 \cup (V_2 \setminus V(C_2))]$ contains a spanning cycle $C_1$ of length $\ell$ which contains $v$. This completes the proof.
\end{proof}


The following lemma guarantees the existence of a suitable almost dominating set in a semicomplete digraph which plays a crucial role to construct robust linkage structures.

\begin{LEM}\label{lem:dom}
Let $\ell\in \mathbb{N}$ and let $D$ be a digraph with $\delta(D) \geq n-\ell$.
For each vertex $x\in V(D)$ and $c\in \mathbb{N}$, there exists sets $A,B\subseteq V(D)$ such that the 
following hold:
 \begin{enumerate}[label=(\roman*)]
\item $|A|,|B|\leq c$ and $D[A]$ contains a spanning path from $x$ and $D[B]$ contains a spanning path to $x$.
\item $|V(D)\setminus \bigcup_{v\in A} N^{-}_D(v)| \leq  2^{1-c} d_D^+(x) + 2\ell$ and
$|V(D)\setminus \bigcup_{v\in B} N^{+}_D(v)| \leq  2^{1-c} d_D^-(x)+ 2\ell$.
\end{enumerate}
\end{LEM}
\begin{proof}
Suppose that for some $i\in [c]$, we have chosen a path $(x=v_1,\dots, v_i)$ of $D$ such that
$$|V(D)\setminus \bigcup_{j=1}^{i} N^{-}_D(v_j)| \leq  2^{1-i} d_D^+(x) + (2-2^{1-i})\ell.$$
Note that such a path exists for $i=1$. 
We will either find a desired set $A$, or extend the path by adding a new vertex. Let 
$$U:= V(D)\setminus \bigcup_{j=1}^{i} N^{-}_D(v_j) \enspace \text{and}\enspace U' := U\setminus N^+_D(v_i),$$ 
then we have $|U'|\leq \ell$ as $U'\subseteq V(D)\setminus N_D(v_i)$.
If $U \setminus U' = \emptyset$, then $|U| = |U'| \leq \ell$ and $A=\left \{v_1,\dots,v_i \right \}$ is a desired set and we are done. Otherwise,
we choose a vertex $v_{i+1} \in U\setminus U'$ with maximum in-degree in $D[U \setminus U']$.
Then it is easy to see that $v_{i+1}$ has at least $(|U\setminus U'|-\ell)/2$ in-neighbors in $U\setminus U'$. Then the set  $V(D)\setminus \bigcup_{j=1}^{i+1}N^{-}_D(v_j)$ is included in the union of $U'$ and the set of vertices in $U \setminus U'$ that are not in-neighbors of $v_{i+1}$. Thus we have
\begin{align*}
|V(D)\setminus \bigcup_{j=1}^{i+1}N^{-}_D(v_j) | &\leq |U'| + \left (|U\setminus U'| - \frac{1}{2}(|U\setminus U'| - \ell) \right ) = |U|/2 + \ell/2+ |U'|/2 \\
& \leq 
2^{-i}d_D^+(x) + (2- 2^{-i})\ell=
2^{1-(i+1)}d_D^+(x) + (2 - 2^{1-(i+1)})\ell.
\end{align*}
As $v_{i+1}\in U\setminus U' = U\cap N^+_D(v_i)$, $(v_1,\dots, v_{i+1})$ forms a path of $D$. By repeating this $c$ times, we obtain the desired set $A$.
Similarly, we can obtain $B$ by symmetry.
\end{proof}
We frequently use the following lemma which is a combined reformulation of Claim~3.1 and Lemma~3.4 in \cite{kang2017sparse}. In \cite{kang2017sparse}, it is only stated for digraphs with at least $k$ vertices. However, the lemma also holds for digraphs with less than $k$ vertices as we can simply take $A=B=V(D)$ in the case.
\begin{LEM}[Kang, Kim, Kim, and Suh~\cite{kang2017sparse}]\label{lem:sparse}
Let $k \in \mathbb{N}$ and $D$ be a digraph such that $\delta(D) \geq n-\ell$.
Then there exist two sets $A,B\subseteq V(D)$ satisfying the following.
 \begin{enumerate}[label=(\roman*)]
\item $|A|,|B|\leq 2k+ \ell -2 $.
\item For any $w\in V(D)$, both $(w,B)$ and $(A,w)$ are $k$-connected in $D$.
\end{enumerate}
\end{LEM}

Recall that for an integer $k \geq 1$, a digraph $D$ is \emph{$k$-linked} if $|V(D)| \geq 2k$ and for any $k$ (not 
necessarily distinct) ordered pairs $(x_1,y_1),\dots, (x_k,y_k)$ of vertices of $D$, there exist $k$ 
distinct internally vertex-disjoint paths $P_1,\dots, P_k$ such that  $P_i$ is a path from $x_i$ to $y_i$ 
such that $V(P)\cap \{x_1 , \dots , x_k , y_1 , \dots , y_k\} = \{x_i, y_i\}$ for each $i\in [k]$. Pokrovskiy~\cite{pokrovskiy2015highly} proved that highly connected tournaments are highly linked.

\begin{THM}[Pokrovskiy~\cite{pokrovskiy2015highly}]\label{thm:link}
For each $k\in \mathbb{N}$, every strongly $452k$-connected tournament is $k$-linked.
\end{THM}

Theorem~\ref{thm:link} can be extended to the following corollary. We omit the proof here, because it can be proved by the almost same proof as in \cite{pokrovskiy2015highly} with obvious modifications using Lemma~\ref{lem:dom}.

\begin{COR}\label{cor:digraph_link}
For all $k,\ell\in \mathbb{N}$, let $D$ be a strongly $(452k + 188 \ell)$-connected digraph with $\delta(D) \geq n-\ell$. Then $D$ is $k$-linked.
\end{COR}

The following is the main ingredient in the proof of both Theorems~\ref{thm:cycle}~and~\ref{thm:balanced}. 
Lemma~\ref{lem:main} states that every highly connected tournament-like digraph $D$ contains pairwise disjoint vertex sets $W_1,\dots, W_t$ such that for each $j\in [k]$, $T[W_j]$ is strongly $k$-connected, and each vertex $u$ outside $\bigcup W_i$ can be added to $W_j$ for many $j\in [t]$ while preserving the connectivity of $D[W_j]$. To be more precise, for each $u\notin \bigcup W_i$ there exists a set $I_u \subseteq [t]$ of indices such that $T[W_j\cup U]$ is strongly $k$-connected for any set $U$ of vertices as long as $j\in I_u$ for all $u\in U$.
These sets $W_j$ and the relationship between the sets $W_1,\dots, W_t$ and vertices in $V(D)\setminus \bigcup W_i$ provides very useful linkage structures in tournament-like digraphs. We prove it in Section~\ref{sec:balanced}.

\begin{LEM}\label{lem:main}
Let $k, t, \ell, m, n, q\in \mathbb{N}$ with $t,m\geq 2$ and $q \geq 1$. 
Suppose that $D$ is an $n$-vertex strongly $10^{8} q k^2 \ell (k+\ell)^2  tm^2 \log(m)$-connected digraph with $\delta(D) \geq n- \ell$ and $Q_1 , \dots , Q_t \subseteq V(D)$ are $t$ disjoint sets with $|Q_i| \leq q$ for each $i \in [t]$. Then there exist disjoint sets $W_1 , \dots , W_t \subseteq V(D)$ that satisfy the following for all $i\in [t]$ and $w\in V_0$, where $V_0:=V(D) \setminus \bigcup_{j=1}^{t}W_j$.
\begin{enumerate}[label=\text{\rm (A\arabic*)}]
\item \label{A5}$Q_i \subseteq W_i$.
\item \label{A2} $D[W_i]$ is strongly $k$-connected.

\item \label{A1} $|W_i| \leq \frac{n}{50 mt}$.
\item \label{A3}$\big|\big\{i'\in [t] :|N_D^{+}(w)\cap W_{i'}|\geq k \text{ and }
|N_D^{-}(w)\cap W_{i'}| \geq k \big\}\big|\geq (1- \frac{1}{30m^{1/2}k})t$.
\item \label{A4}$\big|\big\{w' \in V_0 :|N_D^{+}(w')\cap W_i|\geq k \text{ and }
|N_D^{-}(w')\cap W_i| \geq k  \big\}\big|\geq (1- \frac{1}{10^5 k^4m^2})|V_0|$.
\end{enumerate}
\end{LEM}

\subsection{An example}
Here, we prove Proposition~\ref{prop:example} showing that some lower bound on $a_i$ in Theorem~\ref{thm:balanced} is necessary.

\begin{proof}[Proof of Proposition~\ref{prop:example}]
Let $m\geq 2$ be an integer such that $n= \binom{s+1}{2}m + 2s +1 + s'$ with $1\leq s' < \binom{s+1}{2}$.
Let $$V:= \{  x_s,\dots, x_1, y_1,\dots, y_s\} \cup \{z^i_{j,\ell}: i\in [m], j\in [s], \text{ and } \ell\in[j]\} \cup \{w^*, w_1,\dots, w_{s'} \}.$$
We define an ordering $<$ on $V$ as follows.
\begin{itemize}
\item $x_s< \dots< x_1< w^* < w_1<\dots < w_{s'} < y_1 < \dots, <y_s$.
\item For all $i\in [m], j\in [s]$, and $\ell\in [j]$, we have
$ w^* < z^i_{j,\ell} < w_{1}$.
\item For all $i,i'\in [m], j,j'\in [s]$, $\ell\in [j]$, and $\ell'\in [j']$, we have 
$z^i_{j,\ell} < z^{i'}_{j',\ell'}$ if and only if $(i,j,\ell)<(i',j',\ell')$. 
\end{itemize}
Let 
$$E:=\Big\{\overrightarrow{ z^1_{j,\ell} x_j}, \overrightarrow{y_j z^m_{j,\ell}} : j\in [s], \ell\in [j] \Big\}
\text{ {\Large $\cup$} } \Big\{ \overrightarrow{z^{i+1}_{j,\ell}z^i_{j,\ell}} : (i,j)\in [m-1]\times [s], \ell\in [j]\Big\}.$$
See Figure~1 for an illustration of $E$. Let $T$ be a tournament with $V(T)=V$ and 
$$E(T):= E\cup \{ \overrightarrow{uv} \in V(T)\times V(T) : u<v \text{ and } \overrightarrow{vu}\notin E \}.$$
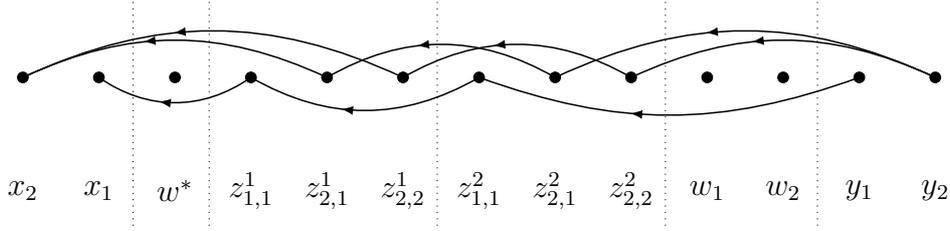
\begin{figure}\label{fig:1}
\centering
\begin{tikzpicture}[scale=1]


\filldraw[fill=black] (1,0) circle (2pt);
\filldraw[fill=black] (2,0) circle (2pt);
\filldraw[fill=black] (3,0) circle (2pt);
\filldraw[fill=black] (4,0) circle (2pt);
\filldraw[fill=black] (5,0) circle (2pt);
\filldraw[fill=black] (6,0) circle (2pt);
\filldraw[fill=black] (7,0) circle (2pt);
\filldraw[fill=black] (8,0) circle (2pt);
\filldraw[fill=black] (9,0) circle (2pt);
\filldraw[fill=black] (10,0) circle (2pt);
\filldraw[fill=black] (11,0) circle (2pt);
\filldraw[fill=black] (12,0) circle (2pt);
\filldraw[fill=black] (13,0) circle (2pt);

\draw[dotted] (2.45,1) to (2.45,-2);

\draw[dotted] (3.45,1) to (3.45,-2);

\draw[dotted] (6.45,1) to (6.45,-2);
\draw[dotted] (9.45,1) to (9.45,-2);

\draw[dotted] (11.45,1) to (11.45,-2);

\node at (1,-1.5) {\large $x_2$};
\node at (2,-1.5) {\large $x_1$};
\node at (3,-1.5) {\large $w^*$};
\node at (4,-1.5) {\large $z^{1}_{1,1}$};
\node at (5,-1.5) {\large $z^{1}_{2,1}$};
\node at (6,-1.5) {\large $z^{1}_{2,2}$};
\node at (7,-1.5) {\large $z^{2}_{1,1}$};
\node at (8,-1.5) {\large $z^{2}_{2,1}$};
\node at (9,-1.5) {\large $z^{2}_{2,2}$};
\node at (10,-1.5) {\large $w_1$};
\node at (11,-1.5) {\large $w_2$};
\node at (12,-1.5) {\large $y_1$};
\node at (13,-1.5) {\large $y_2$};

\draw[middlearrow={latex}, line width=0.2mm] (13,0) to[bend right=25] (9,0);
\draw[middlearrow={latex}, line width=0.2mm] (13,0) to[bend right=25] (8,0);
\draw[middlearrow={latex}, line width=0.2mm] (12,0) to[bend left=20] (7,0);

\draw[middlearrow={latex}, line width=0.2mm] (9,0) to[bend right=30] (6,0);
\draw[middlearrow={latex}, line width=0.2mm] (8,0) to[bend right=30] (5,0);

\draw[middlearrow={latex}, line width=0.2mm] (7,0) to[bend left=30] (4,0);

\draw[middlearrow={latex}, line width=0.2mm] (6,0) to[bend right=25] (1,0);
\draw[middlearrow={latex}, line width=0.2mm] (5,0) to[bend right=25] (1,0);

\draw[middlearrow={latex}, line width=0.2mm] (4,0) to[bend left=35] (2,0);

\end{tikzpicture}
\caption{Ordering of vertices from left to right and edges in $E$ when $s=m=s'=2$.}
\end{figure}
Let $X:=\{x_1,\dots, x_s\}$ and  $Y:=\{y_1,\dots,y_s\}$,
and for each $i\in [m]$, let 
$
Z^i:=\{ z^i_{j,\ell} :j\in [s], \ell\in [j]\}.$
For all $j\in [k]$ and $\ell\in [j]$, we let $P^{j,\ell} :=  (z^{m}_{j,\ell},\dots, z^{1}_{j,\ell})$.
Then $\{P^{j,\ell} : j\in [k] \text{ and } \ell\in [j]\}$ forms a collection of internally vertex-disjoint paths.

Now we prove that $T$ is strongly $s$-connected.
First, for a vertex $v\in V(T)$, either $v\in Y$ or it has at least $s$ out-neighbors in $Y\cup\{w_{s'}\}$. As $Y\subseteq N^+_T(w_{s'})$, we conclude that $(v,Y)$ is $s$-connected in $T$. For each $j\in [s]$, 
\begin{align*}
&\left\{ \big(y_j,P^{j,j},x_j\big), \big(y_j,y_{j+1},P^{j+1,1},x_{j+1}\big), \dots, \big(y_j,y_{s},P^{s,1},x_{s}\big)\right\}\\
 &\cup
\left\{\big(y_j,z_{j,1}^m, z_{j,1}^{m-1},P^{1,1}, x_{1}\big), \dots, \big(y_j,z_{j,j-1}^m, z_{j,j-1}^{m-1},P^{j-1,1}, x_{j-1}\big)\right\}
\end{align*}
forms a collection of $s$ paths from $y_j$ to $X$, where they intersect only at $y_j$. See Figure~2 for an illustration.
Thus, $(y_j, X)$ is $s$-connected in $T$. Together with Lemma~\ref{lem:glue}, this implies that for any $v\in V(T)$, the pair $(v,X)$ is $s$-connected in $T$.

\begin{figure}\label{fig:1}
\centering
\begin{tikzpicture}[scale=0.8]


\filldraw[fill=black] (0,0.5) circle (2pt);
\filldraw[fill=black] (1,0) circle (2pt);
\filldraw[fill=black] (2,-0.5) circle (2pt);
\filldraw[fill=black] (3,0) circle (2pt);

\draw[dotted] (2.45,2) to (2.45,-2.5);

\draw[dotted] (3.45,2) to (3.45,-2.5);

\draw[dotted] (9.45,2) to (9.45,-2.5);

\draw[dotted] (15.45,2) to (15.45,-2.5);
\draw[dotted] (17.45,2) to (17.45,-2.5);

\filldraw[fill=black] (4,-0.5) circle (2pt);
\filldraw[fill=black] (5,-0.5) circle (2pt);
\filldraw[fill=black] (6,0) circle (2pt);
\filldraw[fill=black] (7,0.5) circle (2pt);
\filldraw[fill=black] (8,0) circle (2pt);
\filldraw[fill=black] (9,0) circle (2pt);
\filldraw[fill=black] (10,-0.5) circle (2pt);
\filldraw[fill=black] (11,-0.5) circle (2pt);
\filldraw[fill=black] (12,0) circle (2pt);
\filldraw[fill=black] (13,0.5) circle (2pt);
\filldraw[fill=black] (14,0) circle (2pt);
\filldraw[fill=black] (15,0) circle (2pt);
\filldraw[fill=black] (16,0) circle (2pt);
\filldraw[fill=black] (17,0) circle (2pt);
\filldraw[fill=black] (18,0) circle (2pt);
\filldraw[fill=black] (19,0) circle (2pt);
\filldraw[fill=black] (20,0.5) circle (2pt);

\node at (0,-2) {\large $x_3$};
\node at (1,-2) {\large $x_2$};
\node at (2,-2) {\large $x_1$};
\node at (1, 2) {\large $X$};

\node at (3,-2) {\large $w^*$};
\node at (4,-2) {\large $z^{1}_{1,1}$};
\node at (5,-2) {\large $z^{1}_{2,1}$};
\node at (6,-2) {\large $z^{1}_{2,2}$};
\node at (7,-2) {\large $z^{1}_{3,1}$};
\node at (8,-2) {\large $z^{1}_{3,2}$};
\node at (9,-2) {\large $z^{1}_{3,3}$};

\node at (6.5, 2) {\large $Z^1$};

\node at (10,-2) {\large $z^{2}_{1,1}$};
\node at (11,-2) {\large $z^{2}_{2,1}$};
\node at (12,-2) {\large $z^{2}_{2,2}$};
\node at (13,-2) {\large $z^{2}_{3,1}$};
\node at (14,-2) {\large $z^{2}_{3,2}$};
\node at (15,-2) {\large $z^{2}_{3,3}$};

\node at (12.5, 2) {\large $Z^2$};

\node at (16,-2) {\large $w_1$};
\node at (17,-2) {\large $w_2$};
\node at (18,-2) {\large $y_1$};
\node at (19,-2) {\large $y_2$};
\node at (20,-2) {\large $y_3$};
\node at (19, 2) {\large $Y$};

\draw[middlearrow={latex}, line width=0.2mm] (19,0) to[bend right=25] (12,0);
\draw[middlearrow={latex}, line width=0.2mm] (12,0) to[bend right=25] (6,0);
\draw[middlearrow={latex}, line width=0.2mm] (6,0) to[bend right=25] (1,0);

\draw[middlearrow={latex}, line width=0.2mm] (19,0) to (20,0.5);
\draw[middlearrow={latex}, line width=0.2mm] (20,0.5) to[bend right=25] (13,0.5);
\draw[middlearrow={latex}, line width=0.2mm] (13,0.5) to[bend right=25] (7,0.5);
\draw[middlearrow={latex}, line width=0.2mm] (7,0.5) to[bend right=25] (0,0.5);

\draw[middlearrow={latex}, line width=0.2mm] (19,0) to[bend left=25] (11,-0.5);
\draw[middlearrow={latex}, line width=0.2mm] (11,-0.5) to[bend left=25] (5,-0.5);
\draw[middlearrow={latex}, line width=0.2mm] (5,-0.5) to (10,-0.5);
\draw[middlearrow={latex}, line width=0.2mm] (10,-0.5) to[bend left=25] (4,-0.5);
\draw[middlearrow={latex}, line width=0.2mm] (4,-0.5) to[bend left=25] (2,-0.5);

\end{tikzpicture}
\caption{$s$ paths from $y_2$ to $X$ when $s=3, m=2$ and $s'=2$. }
\end{figure}
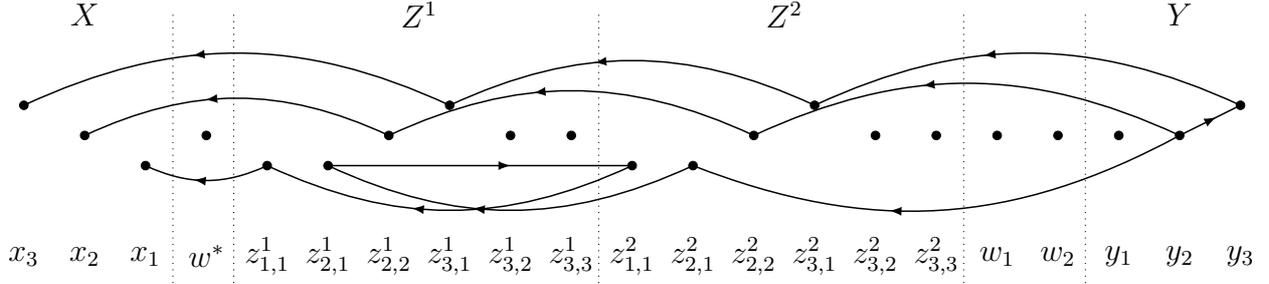

Similarly, for a vertex $v\in V(T)$, either $v\in X$ or it has at least $s$ in-neighbors in $X\cup\{w^*\}$. As $X\subseteq N^-_T(w^*)$, we conclude that $(X,v)$ is $s$-connected in $T$. For each $j\in [k]$, 
\begin{align*}
&\left\{ \big(y_j,P^{j,j},x_j\big), \big(y_{j+1},P^{j+1,1},x_{j+1},x_j\big), \dots, \big(y_{s},P^{s,1},x_{s},x_j\big)\right\}\\
 &\cup
\left\{\big(y_1,P^{1,1},z^{2}_{j,1}, z^{1}_{j,1}, x_{j}\big), \dots, \big(y_{j-1},P^{j-1,1},z^{2}_{j,j-1}, z^{1}_{j,j-1}, x_{j}\big)\right\}
\end{align*}
forms a collection of $s$ paths from $Y$ to $x_j$, where they intersect only at $x_j$. Thus $(Y,x_j)$ is $s$-connected in $T$. Together with Lemma~\ref{lem:glue}, this implies that for any $v\in V(T)$, the pair $(Y,v)$ is $s$-connected in $T$.

Hence, for any $S\subseteq V(T)$ with $|S|\leq k-1$, there exists a path 
$P^X$ from $u$ to $X$ and a path $P^Y$ from $Y$ to $v$.
Since we have $\overrightarrow{xy} \in E(T)$ for all $x\in X$ and $y\in Y$, 
$P^X \cup P^Y$ contains a path from $u$ to $v$ in $T\setminus S$.
This shows that $T$ is strongly $s$-connected.

Let $P$ be a path from a vertex $y \in Y$ to a vertex $x \in X$. For every $i\in [m]$, there is no edge from $Y\cup \bigcup_{j=i+1}^{m}Z^j$ to 
$X\cup \bigcup_{j=1}^{i-1} Z^j$ in $T$. Hence, for each $i\in [m]$, $Z^i$ intersects every path from $Y$ to $X$ in $T$. Since $Z^1,\dots, Z^m$ are pairwise vertex-disjoint, it follows that
\begin{eqnarray}\label{eqn:path}
|\Int(P)| \geq \sum_{i\in [m]} |V(P)\cap Z^i|\geq m,
\end{eqnarray}
implying that $P$ has length at least $m+1 \geq \binom{s+1}{2}^{-1}(n - 2s)$. This proves that $T$ has diameter at least $\binom{s+1}{2}^{-1}(n - 2s)$.

Let $T'$ be a strongly $k$-connected subtournament of $T$. Let $x,y $ be the minimum and maximum elements in $V(T')$ with respect to the order $<$, respectively.
Since every vertex $v\in V(T)\setminus X$ satisfies 
$$|N^-_T(v)\cap \{ v'\in V(T): v' >v\}| \leq 1,$$ 
if $x\notin X$, then $|N^-_T(x)\cap V(T')| \leq |N^-_T(x)\cap \{ v'\in V(T): v' >x\}| \leq 1$, contradicting that $T'$ is strongly $k$-connected with $k\geq 2$. Hence $x \in X$. Similarly, it follows that $y\in Y$. Since $T'$ is strongly $k$-connected, by Menger's theorem, there are $k$ internally vertex-disjoint paths $P_1,\dots, P_{k}$ from $y$ to $x$ in $T'$. By~\eqref{eqn:path}, $|\Int(P_j)| \geq m$ for every $j \in [k]$, and thus $|V(T')| \geq \sum_{j=1}^{k}|\Int(P_{j})| + |\left \{x,y \right \}| \geq km+2 \geq k \binom{s+1}{2}^{-1} n-k-2$. This finishes the proof.
\end{proof}

\section{Proof of Lemma~\ref{lem:main}}\label{sec:balanced}

\noindent {\bf Outline of the proof}. We first sketch the idea of the proof. As in \cite{kuhn2016cycle}, our proof starts with constructing robust linkage structures. However, we introduce more involved ideas and techniques in order to obtain the connectivity bound linear in $t$.

We aim to find $t$ disjoint subsets of vertices $W_1 , \dots , W_t$ satisfying all the conditions~\ref{A5}--\ref{A4} of Lemma~\ref{lem:main}. First, let us consider the following ideal scenario. Assume
$Q_1 = \dots = Q_t = \emptyset$ and that there are small in-dominating sets $A_{i,j}$ (i.e. every vertex outside $A_{i,j}$ has an out-neighbor in $A_{i,j}$) and small out-dominating sets $B_{i,j}$ (i.e. every vertex outside $B_{i,j}$ has an in-neighbor in $B_{i,j}$) for each $(i,j) \in [t] \times [k]$, so that all sets are pairwise disjoint and each $D[A_{i,j}]$ and each $D[B_{i,j}]$ contain a spanning path.
As $D$ is highly connected, provided that all these dominating sets are small enough, one can use Corollary~\ref{cor:digraph_link} to find $kt$ vertex-disjoint paths $P_{i,j}$ from the sink of the spanning path in $D[A_{i,j}]$ to the source of the spanning path in $D[B_{i,j}]$, where the path $P_{i,j}$ intersects $\bigcup_{i \in [t]}\bigcup_{j\in[k]}(A_{i,j} \cup B_{i,j})$ only at its ends.

We will later construct $W_1,\dots, W_t \subseteq V(D)$ in such a way that each set $W_i$ contains $\bigcup_{j \in [k]} (A_{i,j}\cup P_{i,j}\cup B_{i,j})$. This will guarantee the following property: for any vertex $u , v \in W_i\setminus \bigcup_{j \in [k]} (A_{i,j}\cup P_{i,j}\cup B_{i,j})$, there are $k$ edges from $u$ to $\bigcup A_{i,j}$ and $k$ edges from $\bigcup B_{i,j} $ to $v$ and these together with $k$ paths $P_{i,j}$ in $W_i$ gives $k$ internally vertex-disjoint paths from $u$ to $v$. 
This property will be later useful to ensure that each $W_i$ satisfies \ref{A2}.

However, we cannot hope for this ideal case, due to several issues. The following are some major issues complicating the proof. Figure~\ref{fig: Z} drawn at the beginning of {\bf Step~3} will be helpful to understand the structure we construct in {\bf Step~1--Step~3}.

\begin{itemize}
\item Because dominating sets may be large in general, one cannot apply Corollary~\ref{cor:digraph_link} to obtain the desired paths $P_{i,j}$ internally disjoint from the dominating sets. Instead of in/out-dominating sets, we will use Lemma~\ref{lem:dom} to construct small almost in-dominating sets $A_{i,j}$ and almost out-dominating sets $B_{i,j}$ that dominate most of the vertices outside $\bigcup_{i \in [t]}\bigcup_{j\in[k]}(A_{i,j} \cup B_{i,j})$.
Moreover, we will assign $3k$, instead of $k$, almost in/out-dominating sets to the set $W_i$ and define vertices ``exceptional'' (denoted by $E_i^S$ in {\bf Step 2}) if they are not dominated by at least $k$ of those almost dominating sets. This will enable us to control the number of ``exceptional'' vertices (see~\eqref{eq: E size}) and help us to obtain \ref{A4}. We will further elaborate on them in {\bf Step 1} and {\bf Step 2}.

\item Corollary~\ref{cor:digraph_link} does not provide any information on the length of each path. Hence some paths $P_{i,j}$ may be very long and thus $W_i$ may violate~\ref{A1}. In order to control the length of the paths, we build disjoint structures $\bigcup_{j\in [3k]} (A_{i,j}\cup P_{i,j}\cup B_{i,j})$ for each $i\in [h]$ with $h=1200k (k+\ell) tm$ instead of $t$.
This choice of $h$ ensures that many of them are small enough for~\ref{A1} (see Claim~\ref{cl: path short}). We will assign $t$ appropriate structures to the sets $W_1,\dots, W_t$ and discard the vertices in the rest structures (see Claim~\ref{lem:indep}). 

However, this may cause other problems, as these discarded vertices might not be in/out-dominated by most of the remaining dominating sets.
 To obtain \ref{A3} at the end,  we will distribute the discarded vertices to $W_1,\dots, W_t$ nicely with guaranteeing \ref{A2} and \ref{A1}. We will also further elaborate on them in {\bf Step 3}.

\item 
After {\bf Step 3}, there still could be some vertices not satisfying~\ref{A3} and not assigned to $W_i$ (the sets of these vertices are denoted by $\widetilde{F}^A$ and $\widetilde{F}^B$ in {\bf Step 4}). So we will distribute these vertices into one of the sets $W_i$. Showing that the number of such problematic vertices is small, we will be able to distribute them nicely so that both \ref{A2} and \ref{A1} still hold. We will further elaborate on them in {\bf Step 4}.
\end{itemize}

 \vspace{0.3cm}

\noindent {\bf Step 1. Construction of almost dominating sets.}
Let $D$ be a strongly $10^{8} q k^2 \ell (k+\ell)^2  tm^2 \log(m)$-connected $n$-vertex digraph with $\delta(D)\geq n-\ell$. Let $V:=V(D)$ and $h:= 1200k(k+\ell)t m$ and $c:= 31 + \lceil 5 \log (k\ell m) \rceil$.
 As $D$ is strongly $10^{8} q k^2 \ell (k+\ell)^2  tm^2 \log(m)$-connected, it follows that
\begin{align}\label{eq: n min deg}
\min\{ n, \delta^+(D), \delta^-(D)\} \geq 10^{8} q k^2 \ell (k+\ell)^2 tm^2\log(m).
\end{align}
By our assumption, $t$ pairwise disjoint subsets $Q_1 , \dots , Q_t \subseteq V$ are given and satisfy $|Q_i| \leq q$ for each $i \in [t]$. Let 
$$Q := \bigcup_{j\in [t]}Q_j, \enspace V':=V\setminus Q, \enspace D' := D \setminus Q \enspace \text{and}\enspace n' := |V(D')|.$$ Then $D'$ is strongly $9 \cdot 10^{7} q k^2 \ell (k+\ell)^2 tm^2\log(m)$-connected and
\begin{align}\label{eq: n' min deg}
\min\{ n', \delta^+(D'), \delta^-(D')\} \geq 9 \cdot 10^{7} q k^2 \ell (k+\ell)^2 tm^2\log(m).
\end{align}

First, we choose vertices of small out-degrees and small in-degrees, and then we will construct almost dominating sets $A_{i,j}$ and $B_{i,j}$ using these vertices. We plan to allocate the vertices in $A_{i,j}\cup B_{i,j}$ into $W_i$ later.
Using these vertices to construct almost dominating sets, we can  control the size of the set of `exceptional' vertices (vertices that are not dominated) in terms of degrees of vertices (see~\ref{AB3}). This will ensure that any exceptional vertex has many non-exceptional out-neighbors or  many non-exceptional in-neighbors. If we put these neighbors along with $v$ to $W_i$, then $v$ will have enough neighbors to reach (or to be reached from) non-exceptional vertices.
We construct such sets for all
$(i,j) \in [h]\times [3k]$ instead of just $(i,j)\in [t]\times [k]$ so that we have more flexibility later. For example, we can discard some sets later if necessary as $h> t$.

Let $\big\{ x_{i,j} : (i,j)\in [h]\times[3k]\big\}$ and
$\big\{ y_{i,j} : (i,j) \in [h]\times[3k]\big\}$ be two disjoint subsets of $V'$, each consisting of $3kh$ vertices with smallest out-degrees/in-degrees in 
$D'$, respectively. Since $n'\geq 6kh$ by~\eqref{eq: n' min deg}, these sets exist. Let
$$V_{\rm low}:= \big\{ x_{i,j} : (i,j)\in [h]\times[3k]\big\} \cup
\big\{ y_{i,j} : (i,j) \in [h]\times[3k]\big\}.$$
\begin{align*}
\delta_0^+:= \min_{w \in V'\setminus V_{\rm low}}d_{D'}^+(w) \enspace \text{and} \enspace \delta_0^- := \min_{w \in V'\setminus V_{\rm low}}d_{D'}^-(w).
\end{align*}
By symmetry, we may assume that 
\begin{align}\label{eq: symmetry assume}
\delta_0^+ \geq \delta_0^-.
\end{align}
The other case follows from a symmetric argument. Since $|V_{\rm low}| \leq 6kh$, by \eqref{eq: n min deg} for every $(i,j)\in [h]\times[3k]$, there are subsets $X_{i,j}$ and $Y_{i,j}$ of $V'$ satisfying the following for all $(i,j)\in [h]\times [3k]$.\vspace{0.1cm}
\begin{enumerate}[label=\text{(XY\arabic*)}]
\item \label{XY1} The sets $X_{1,1} , \dots , X_{h,3k} , Y_{1,1} , \dots , Y_{h,3k}$  are pairwise disjoint.
\item \label{XY2} $x_{i,j}\in X_{i,j}$, $y_{i,j}\in Y_{i,j}$ and $|X_{i,j}|= |Y_{i,j}| = 2k+1.$
\item \label{XY3} $x_{i,j}$ has $k$ out-neighbours and $k$ in-neighbours in $X_{i,j}$ and $y_{i,j}$ has $k$ out-neighbours and $k$ in-neighbours in $Y_{i,j}$, respectively.
\end{enumerate}\COMMENT{We have `$\geq k$' instead of `$=k$' because there could be cycles of length two in $D'$.} \vspace{0.1cm}
Let $V_{\rm low}^*:= \bigcup_{(i,j)\in [h]\times [3k]} (X_{i,j}\cup Y_{i,j})$.
We define $A_{i,j}$ and $B_{i,j}$ for each  $(i,j)\in [h]\times[3k]$ as follows.
For each $(i,j)\in [h]\times[3k]$ in lexicographic order, we repeatedly apply Lemma~\ref{lem:dom} to 
a digraph
$D^X_{i,j}:= D' \setminus \big((V^*_{\rm low}\setminus\{x_{i,j}\})\cup \bigcup_{(i',j') < (i,j)} A_{i',j'} \big)$ and the vertex $x_{i,j}$ with parameter $c$ to obtain a set $A_{i,j}$ and a vertex $a_{i,j}$ where $D'[A_{i,j}]$ has a spanning path from $x_{i,j}$ to $a_{i,j}$. Again, for $(i,j)\in [h]\times[3k]$ in lexicographic order, we repeatedly apply Lemma~\ref{lem:dom} to a digraph
$D^Y_{i,j}:=D' \setminus \big((V^*_{\rm low}\setminus\{y_{i,j}\})\cup \bigcup_{(i',j')\in [h]\times[3k]} A_{i',j'}\cup 
\bigcup_{(i',j') < (i,j) } B_{i',j'} \big)$ and the vertex $y_{i,j}$ with parameter $c$ to obtain a set $B_{i,j}$ and a vertex $b_{i,j}$ where $D'[B_{i,j}]$ has a spanning path from $b_{i,j}$ to $y_{i,j}$. 
Then we obtain pairwise disjoint sets $A_{1,1},\dots, A_{h,3k}, B_{1,1},\dots, B_{h,3k} \subseteq V'$ and 
vertices $a_{1,1},\dots, a_{h,3k},b_{1,1},\dots, b_{h,3k}$.
For each $i\in [h]$, let
$$C_i:= \bigcup_{j\in [3k]} (A_{i,j}\cup B_{i,j} \cup X_{i,j}\cup Y_{i,j}) \enspace \text{and} \enspace C:= \bigcup_{i\in [h]} C_i.$$
For each $(i,j)\in [h]\times [3k]$, let
$$E^A(i,j):=V'\setminus \left(C\cup \bigcup_{v\in A_{i,j}} N^-_{D'}(v)\right) \enspace\text{and}
 \enspace E^B(i,j):=V'\setminus \left(C\cup \bigcup_{v\in B_{i,j}} N^+_{D'}(v)\right),$$
 which are the set of vertices in $V' \setminus C$ not in/out-dominated by $A_{i,j}$ and $B_{i,j}$, respectively. Then the following statements hold for every $(i,j)\in  [h]\times[3k]$.
\begin{enumerate}[label=\text{(AB\arabic*)}]
\item \label{AB1} The sets $X_{1,1},\dots,X_{h,3k} , Y_{1,1} , \dots , Y_{h,3k} , A_{1,1} , \dots , A_{h,3k} , B_{1,1} , \dots , B_{h,3k}$ are pairwise disjoint, $1\leq |A_{i,j}|,|B_{i,j}|\leq c$ and $|C_i| \leq 250k^2 + 30k \log(km\ell)$.\footnote{
For \ref{AB1}, 
$|C_i| \leq \sum_{j\in [3k]} (|A_{i,j}|+|B_{i,j}| +|X_{i,j}| + |Y_{i,j}|) \leq 3k(62+ 2\lceil 5 \log (kml) \rceil+ 4k+2) \leq 12k^2+200k + 30k\log(km\ell) \leq  250k^2 + 30k \log(km\ell)$. 
} 

\item \label{AB2} $D'[A_{i,j}]$ contains a spanning path from $x_{i,j}$ to $a_{i,j}$ and $D'[B_{i,j}]$ contains a spanning path from $b_{i,j}$ to $y_{i,j}$.

\item \label{AB3} $\displaystyle |E^A(i,j)| \leq 2^{1-c} d_{D'}^+(x_{i,j}) + 2\ell \stackrel{\eqref{eq: n' min deg}}{\leq} \frac{\delta_0^+}{10^7 k^2 (k+\ell)^2m^2}$
and  \newline $\displaystyle |E^B(i,j)| \leq 2^{1-c} d_{D'}^-(y_{i,j})+2\ell \stackrel{\eqref{eq: n' min deg}}{\leq}   \frac{\delta_0^-}{10^7 k^2(k+\ell)^2m^2}$.
\end{enumerate}
Note that we obtain \ref{AB3} because we have $d^+_{D^X_{i,j}}(x_{i,j})\leq d_{D'}^+(x_{i,j})\leq \delta_0^+$ and $d^-_{D^Y_{i,j}}(y_{i,j})\leq d_{D'}^+(y_{i,j})\leq \delta_0^+$.
Note that 
\begin{align}\label{eq: size C}
|C| \stackrel{\ref{AB1}}{\leq}250 k^2 h + 30 kh \log (km\ell) \leq 4\cdot 10^5 k^2(k+\ell)^2 tm \log(m).
\end{align}

Later, we wish to assign vertices in $Q_i$ to some $W_j$. In order for this, we need to ensure that there are many paths from a vertex in $Q_i$ to $W_j$. To ensure this we want to prepare a set $Q'_i$ and a set $Q^*_i$ so that
there are many paths between $Q_i$ and $Q^*_i$ through $Q'_i$, and there are many paths between $Q^*_i$ and $C_j$ for many $j\in [h]$ (See Figure~\ref{fig: Z} for the case when $j=1$).
Note that the reason why we need $Q'_i$ is that the vertices in $Q_i$ might have very low in/out-degree compare to $\delta^-_0$ or $\delta^+_0$. In particular, \eqref{eq: outside F} might not hold for the vertices in $Q_i$, but it holds for the vertices in $Q'_i$. 


\begin{CLAIM}\label{claim:q'}
There exist pairwise disjoint sets $Q_1' , \dots , Q_t' \subseteq V' \setminus C$ satisfying the following for each $i \in [t]$ and $u\in Q_i\cup Q'_i$.
\begin{enumerate}[label=\text{\rm (Q$'$\arabic*)}]
\item\label{Q'0} The sets $C , Q_1 , \dots , Q_t , Q_1' , \dots , Q_t'$ are pairwise disjoint.
\item\label{Q'1} $|Q_i'| \leq 4k(k+\ell)$. 
\item\label{Q'2} Both $(u, Q_i')$ and $(Q_i',u)$ are $k$-connected in $D[Q_i \cup Q_i']$.
\end{enumerate}
\end{CLAIM}
\begin{proof}
We iteratively define pairwise disjoint sets $Q_1',\dots, Q'_t$ in order satisfying all \ref{Q'0}--\ref{Q'2}.
 Let $i \in [t]$ and assume that we have defined $Q_1' , \dots , Q_{i-1}'$ satisfying all \ref{Q'0}--\ref{Q'2}. By applying Lemma~\ref{lem:sparse} to $D[Q_i]$, there are two sets $Q^{\rm source}_i, Q^{\rm sink}_i \subseteq Q_i$ such that $|Q^{\rm source}_i|, |Q^{\rm sink}_i|\leq 2k+ \ell-2$ and 
\begin{equation}\label{eq: Qi kconn}
\begin{minipage}[c]{0.9\textwidth}\em
for every $v\in Q_i$, 
both 
$(v,Q^{\rm sink}_i)$ and $(Q^{\rm source}_i,v)$ are $k$-connected in 
$D[Q_i]$.
\end{minipage}
\end{equation}
Note that for each $u\in Q^{\rm source}_i$, we have
\begin{align*}
\big|N_D^-(u)\setminus (Q \cup C \cup \bigcup_{j\in [i-1]} Q_{j}')\big|
\stackrel{\eqref{eq: size C}, \ref{Q'1} }{\geq} \delta^{-}(D) - qt - 4\cdot 10^5 k^2(k+\ell)^2 tm \log(m) - 4kt(k+\ell) \stackrel{\eqref{eq: n min deg}}{\geq} k. 
\end{align*}
Thus for every $u \in Q^{\rm source}_i$, we can choose a set  $\widetilde{Q}_{i,u} \subseteq N_D^-(u) \setminus (Q \cup C \cup \bigcup_{j=1}^{i-1} Q'_{j})$ with $|\widetilde{Q}_{i,u}|= k$. 
Let $\widetilde{Q}_i := \bigcup_{u \in Q^{\rm source}_i}\widetilde{Q}_{i,u}$, then we have $|\widetilde{Q}_i|\leq k(2k+\ell-2)$ and 
\begin{equation}\label{eq: Qsharpi kconn}
\begin{minipage}[c]{0.9\textwidth}\em
for each $u\in Q^{\rm source}_i$, the pair
$(\widetilde{Q}_i , u )$ is $k$-connected in 
$D[Q_i\cup \widetilde{Q}_i]$.
\end{minipage}
\end{equation}
Similarly, for each $u\in Q^{\rm sink}_i$, we have
\begin{align*}
\big|N_D^+(u)\setminus (Q \cup C \cup \bigcup_{j\in [i-1]} Q_{j}')\big|
\stackrel{\eqref{eq: size C}, \ref{Q'1}}{\geq} \delta^{+}(D) - qt - 4\cdot 10^5 k^2(k+\ell)^2 tm \log(m) - 4kt(k+\ell) \stackrel{\eqref{eq: n min deg}}{\geq} k. 
\end{align*}
Thus for every $u \in Q^{\rm sink}_i$, we can choose a set  $\widehat{Q}_{i,u} \subseteq N_D^+(u) \setminus (Q \cup C \cup \bigcup_{j=1}^{i-1} Q'_{j})$ with $|\widehat{Q}_{i,u}|= k$. 
Let $\widehat{Q}_i := \bigcup_{u \in Q^{\rm sink}_i}\widehat{Q}_{i,u}$, then we have $|\widehat{Q}_i|\leq k(2k+\ell-2)$ and 
\begin{equation}\label{eq: Qsharpi kconn_2}
\begin{minipage}[c]{0.9\textwidth}\em
for each $u\in Q^{\rm sink}_i$, the pair
$(u , \widehat{Q}_i )$ is $k$-connected in 
$D[Q_i\cup \widehat{Q}_i]$.
\end{minipage}
\end{equation}
Let $Q_i' := \widetilde{Q_i} \cup \widehat{Q_i}$. 
Then $|Q'_i| \leq 2k(2k+\ell-2) \leq 4k(k+\ell)$, so we prove \ref{Q'1}.
The contruction and \ref{AB1} imply \ref{Q'0}.
 Combining Lemma~\ref{lem:glue} with~\eqref{eq: Qi kconn},~\eqref{eq: Qsharpi kconn}, and~\eqref{eq: Qsharpi kconn_2} it follows that $Q_i'$ satisfies \ref{Q'2}. This completes the proof.
\end{proof}

Let $Q' := \bigcup_{i=1}^{t}Q_i'$. Then by Claim~\ref{claim:q'} we have
\begin{align}\label{eqn:q'}
|Q'| \leq \sum_{i=1}^{t}|Q_i'| \leq 4k(k+\ell)t.
\end{align}

\noindent {\bf Step 2. Defining exceptional vertices and constructing sets $C^*_i$ and $Q^*_i$.} In this step, we define layers of the following two types of ``exceptional" vertices.

For each $i\in [h]$ and $S\in \{A,B\}$, we define
\begin{align}\label{eq: E def}
\begin{split}
E^S_i  &:=  \left\{x \in V' \setminus C \: \colon \: \big|\{ j\in [3k] \colon x \in E^S(i,j)\}\big| \geq k \right\} \text{ and }
\end{split}\\
\begin{split}\label{eq: F def}
F^S & :=\big\{x \in V' \setminus C \: \colon \: 
|\{ i\in [h] \colon x \in E^S_i\}| \geq \frac{t}{100 k(k+\ell)} \big\}.
\end{split}
\end{align}

Note that for any vertex $u \in V' \setminus (C \cup E^A_i \cup E^B_i)$, there are at least $3k - 2k \geq k$ indices $j \in [3k]$ such that there exist an edge from $u$ to $A_{i,j}$ and an edge from $B_{i,j}$ to $u$. Hence, this ensures that if $W \subseteq V(D)$ contains $C_i$ and $D[W]$ is strongly $k$-connected, then $D[W \cup \{u \}]$ is also strongly $k$-connected. For each $j \subseteq [h]$, if $U_j \subseteq V(D)$ contains $C_j$ and $D[U_j]$ is strongly $k$-connected, then for any vertex $u \in V' \setminus (C \cup F^A \cup F^B)$, the digraph $D[U_j \cup \{u \}]$ is strongly $k$-connected for all but at most $\frac{t}{100k(k+\ell)}$ choices $j \in [h]$.

These definitions yield the following upper bounds on the size of each exceptional set, which is independent of $t$. For each $i\in [h]$, we have
\begin{align}
\begin{split}\label{eq: E size}
|E^A_i| &\leq\frac{1}{k} \sum_{j=1}^{3k}|E_A(i,j)| \stackrel{\ref{AB3}}{\leq}
\frac{3\delta_0^+}{10^{7}k^2(k+\ell)^2m^2 }, \enspace \\
|E^B_i| &\leq \frac{1}{k}\sum_{j=1}^{3k}|E_B(i,j)| \stackrel{\ref{AB3}}{\leq} \frac{3\delta_0^-}{10^{7}k^2(k+\ell)^2 m^2 }, 
\end{split}\\
\begin{split}\label{eq: F size}
|F^A| &\leq \frac{100k(k+\ell)}{t}\sum_{i \in [h]}|E^A_i|\stackrel{\eqref{eq: E size}}{\leq} \frac{300k(k+\ell)h \delta_0^+}{10^{7}k^2(k+\ell)^2  m^2 t } \leq
\frac{\delta^+_0}{24 m}\leq
\frac{\delta^+_0}{30}  \enspace \text{and}  \enspace \\
|F^B| &\leq \frac{100k(k+\ell)}{t}\sum_{i \in [h]}|E^B_i|\stackrel{\eqref{eq: E size}}{\leq} \frac{300k(k+\ell)h \delta_0^-}{10^{7}k^2(k+\ell)^2  m^2 t } \leq 
 \frac{\delta^-_0}{30 } \stackrel{\eqref{eq: symmetry assume}}{\leq}  \frac{\delta^+_0}{30 }. 
\end{split}
\end{align}\COMMENT{ 
$\frac{100k(k+\ell)h \delta_0^+}{10^{6}k^2(k+\ell)^2  m^2 t }
\leq \frac{120000 \delta_0^+}{10^{6}  m^2} \leq \frac{\delta^+_0}{30 }$ as $m\geq 2$.
}

Therefore, for every vertex $v\in V'\setminus V_{\rm low}$ and $i\in [h]$, \eqref{eq: E size} with \eqref{eq: F size}, \eqref{eq: n min deg} and \eqref{eq: size C} implies that
\begin{align}\label{eq: outside F}
|N_{D'}^+(v)\setminus (C\cup F^A\cup F^B \cup E^A_i\cup E^B_i)|\geq \frac{9}{10}\delta^+_0 \enspace \text{and} \enspace
|N_{D'}^-(v)\setminus (C\cup F^B\cup E^B_i)| \geq \frac{9}{10}\delta^-_0.
\end{align}

The bounds~\eqref{eq: outside F} allow us to choose non-exceptional neighbours of a vertex $u\in V' \setminus V_{\rm low}$ which has in/out-degree at least $\delta^-_0$ and $\delta^+_0$, respectively. In {\bf Step 3}, we will be often in a position to ``distribute" many vertices into pairwise disjoint  $k$-connected sets, while preserving $k$-connectedness and not increasing the size of each too much. We will use the bound on $|F^S|$ to obtain a desired distribution. We will also use the bound on $|E_i^S|$ to ensure~\ref{A4} of Lemma~\ref{lem:main} at the end.

Let $W \subseteq V(D)$ be a set containing $C_i$. Some exceptional vertices in $(E_i^A \cup E_i^B ) \setminus W$ may not have edges to $A_{i,j}$ or edges from $B_{i,j}$ for many indices $j \in [3k]$, however, we wish to find $k$ internally vertex-disjoint paths from/to each exceptional vertex to/from $W$, respectively. In order to do this, for each exceptional vertex $v \in (E_i^A \cup E_i^B) \setminus W$, we take non-exceptional vertices $v^+_1,\dots, v^+_k \in N^+_D(v)$ and $v^-_1,\dots, v^-_k\in N^-_D(v)$ and allocate all vertices $v,v^+_1,\dots, v^+_k, v^-_1,\dots, v^-_k$ together to $W$.
Then there are $k$ internally vertex-disjoint paths from/to $v$ to/from $W$ through the vertices in $\{ v^+_1 , \dots , v^+_k, v^-_1,\dots, v^-_k \}$, respectively.


In a similar way, when we add vertices in some set $U$ into a set $W$, we will find another small set $U^*$ of vertices and put $U$ together with $U^*$ into $W$ in such a way that there are many paths between vertices in $U\cup U^*$  and non-exceptional vertices in $U^*$. We often use the following Claim~\ref{cl: extension} as a blackbox in order to build such  well connected linkages in {\bf Step 3}. 

Note that the proof of Claim~\ref{cl: extension} looks similar to the proof of Claim~\ref{claim:q'}, however, it uses slightly different arguments due to the additional constraints in~\ref{U*2} that guarantee linkages from/to non-exceptional vertices. While \ref{Q'2} trivially holds for vertices in $Q'_i$, \ref{U*2} is no longer obvious for vertices in $U_i^*$ which also belong to either $E_i^A \cup F^A$ or $E_i^B \cup F^B$.

\begin{figure}
\centering
\begin{tikzpicture}[scale=1]

  \draw[fill=none] (0,0) ellipse [x radius=3,y radius=1];
  \draw (-1,0.96) -- (-1,-0.96);
   \draw (1,0.96) -- (1,-0.96);
  \node at (0,0) {$U_i$};
  \node at (2.2,0) {$U^{\rm sink}_i$};
  \node at (-2,0) {$U^{\rm source}_i$};

\draw[middleupupuparrow={latex}, line width=0.7mm] (-1.25,0) -- (-0.75,0);
\draw[middleupupuparrow={latex}, line width=0.7mm] (0.75,0) -- (1.25,0) ;

 \draw[fill=none] (-2,-1.5) ellipse [x radius=0.9,y radius=0.65];
 \node at (-2.3,-1.5) {$\widetilde{U}_i$};
 
\draw[middleupupuparrow={latex}, line width=0.7mm] (-1.8,-1.1) -- (-1.55,-0.7);

\draw[pattern=crosshatch] (-1.4,-1.5) ellipse [x radius=0.3,y radius=0.3];
\draw[middleupuparrow={latex}, line width=0.7mm] (-1.2,-1.5) -- (-0.55,-1.5);
\draw[middleupupupuparrow={latex}, line width=0.7mm] (-1.9,-1.5) -- (-1.5,-1.5);

\draw[pattern=crosshatch] (1.3,-0.62) ellipse [x radius=0.3,y radius=0.3];
\draw[middleupuparrow={latex}, line width=0.7mm] (1.2,-0.75) -- (0.8,-1.35);

\draw[middleupupuparrow={latex}, line width=0.7mm] (1.6,-0.1) -- (1.3,-0.55);

 \draw[fill=none] (0.2,-1.5) ellipse [x radius=0.9,y radius=0.45];
 \node at (0.2,-1.5) {$\widehat{U}_i$};

\draw[pattern=crosshatch] (3,-1.7) ellipse [x radius=0.3,y radius=0.3];
 \node at (4.1,-1.7) {= $\widetilde{U}^{\rm sink}_i$};

\end{tikzpicture}
\caption{Structure of $U_i^*$ in the proof of Claim~\ref{cl: extension}.}\label{fig: *}
\end{figure}
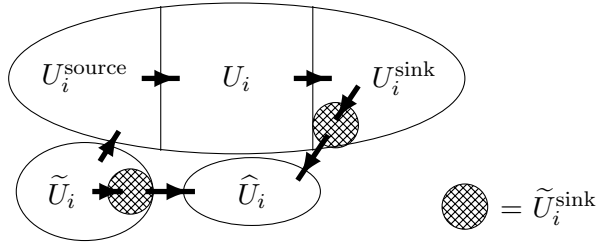

\begin{CLAIM}\label{cl: extension}
Let $U\subseteq V(D')$ with $|U| \leq 10^6 k^2\ell (k+\ell)^2 tm \log(m)$. Let $h'\in [h]$ and $U_1,\dots, U_{h'} \subseteq V(D')\setminus V_{\rm low}$.
Then there exist pairwise disjoint sets $U^*_1,\dots, U^*_{h'} \subseteq V(D')\setminus U$ such that the following hold for all $i\in [h']$ and $u\in U_i\cup U^*_i$.
\begin{enumerate}[label=\text{\rm (U$^*$\arabic*)}]
\item\label{U*1} $|U^*_i| \leq 4k(k+\ell)$.
\item\label{U*2} Both $(u,U^*_i\setminus (C\cup E^A_i\cup F^A) )$ and $(U^*_i\setminus (C\cup E^B_i\cup F^B),u )$ are $k$-connected in $D[U_i\cup U^*_i]$.
\end{enumerate}
\end{CLAIM}
\begin{proof}
We iteratively construct pairwise disjoint sets $U_1^* , \dots , U_{h'}^* \subseteq V(D') \setminus U$ satisfying the claim. Let us assume that for some $i\in [h']$ we have constructed pairwise disjoint sets $U^*_1,\dots, U^*_{i-1}\subseteq V'\setminus U$ satisfying the following.
\begin{enumerate}
\item[(U1)$^{i-1}$] For each $j\in [i-1]$, we have $|U^*_j| \leq 4k(k+\ell)$.
\item[(U2)$^{i-1}$] For each $j\in [i-1]$ and $u\in U_j\cup U^*_j$, both $(u,U^*_j \setminus (C\cup E^A_j\cup F^A) )$ and $(U^*_j \setminus (C\cup E^B_j\cup F^B),u )$ are $k$-connected in $D[U_j\cup U^*_j]$.
\end{enumerate}

Note that the empty collection satisfies (U1)$^{0}$ and (U2)$^{0}$. 
See Figure~\ref{fig: *}. We want to build sets $\widetilde{U}_i$ and $\widehat{U}_i$ such that $(\widetilde{U}_i,u)$ and $(u,\widehat{U}_i)$ are both $k$-connected in $D[U^*_i]$ for all $u\in U^*_i$.

By applying Lemma~\ref{lem:sparse} to $D[U_i]$, there are two sets $U^{\rm source}_i, U^{\rm sink}_i \subseteq U_i$ such that $|U^{\rm source}_i|, |U^{\rm sink}_i|\leq 2k+ \ell-2$ and 
\begin{equation}\label{eq: Ci kconn}
\begin{minipage}[c]{0.9\textwidth}\em
for every $v\in U_i$, 
both 
$(v,U^{\rm sink}_i)$ and $(U^{\rm source}_i,v)$ are $k$-connected in 
$D[U_i]$.
\end{minipage}
\end{equation}
\COMMENT{If $U_i=\emptyset$, then $U^{\rm source}_i = U^{\rm sink}_i=\emptyset$ vacuously satisfies \eqref{eq: Ci kconn}.}
Note that for each $u\in U^{\rm source}_i$, we have
\begin{align*}
\big|N_{D'}^-(u)\setminus (U\cup C\cup E^B_i\cup F^B \cup \bigcup_{j\in [i-1]} U^*_{j})\big|
\stackrel{\substack{\eqref{eq: outside F},\eqref{eq: size C},\\ \text{(U1)$^{i-1}$}} }{\geq} \frac{9}{10}\delta^-_0- 2\cdot 10^6 k^2\ell(k+\ell)^2 tm \log(m)\stackrel{\eqref{eq: n' min deg}}{\geq} k.  
\end{align*}
Thus for every $u \in U^{\rm source}_i$, we can choose a set  $\widetilde{U}_{i,u} \subseteq N_{D'}^-(u) \setminus (U \cup C \cup E^B_i\cup F^B \cup \bigcup_{j=1}^{i-1} U^*_{j})$ with $|\widetilde{U}_{i,u}|= k$. 
Let $\widetilde{U}_i := \bigcup_{u \in U^{\rm source}_i}\widetilde{U}_{i,u}$, then we have $|\widetilde{U}_i|\leq k(2k+\ell-2)$ and 
\begin{equation}\label{eq: Csharpi kconn}
\begin{minipage}[c]{0.9\textwidth}\em
for each $u\in U^{\rm source}_i$, the pair
$(\widetilde{U}_i\setminus (C\cup E^B_i\cup F^B),u )$ is $k$-connected in 
$D[U_i\cup \widetilde{U}_i]$.
\end{minipage}
\end{equation}

Similarly, applying Lemma~\ref{lem:sparse} to $D[\widetilde{U}_i\cup U^{\rm sink}_i]$, there exists a set $\widetilde{U}^{\rm sink}_i \subseteq \widetilde{U}_i\cup U^{\rm sink}_i$ such that
$|\widetilde{U}^{\rm sink}_i|\leq 2k+ \ell-2$ and 
\begin{equation}\label{eq: tilde Ci kconn}
\begin{minipage}[c]{0.9\textwidth}\em
for each $v\in \widetilde{U}_i\cup U^{\rm sink}_i$, the pair
$(v,\widetilde{U}^{\rm sink}_i)$ is $k$-connected in 
$D[\widetilde{U}_i\cup U^{\rm sink}_i]$.
\end{minipage}
\end{equation}
Note that, again for each $u\in \widetilde{U}^{\rm sink}_i$, we have
\begin{align*}
\big|N_{D'}^+(u)\setminus (U\cup C\cup \bigcup_{S\in \{A,B\}} (E^S_i\cup F^S) \cup \bigcup_{j\in [i-1]} U^*_{j})\big|
\stackrel{\substack{\eqref{eq: outside F},\eqref{eq: size C},\\ \text{(U1)$^{i-1}$}} }{\geq} \frac{9}{10} \delta^+_0 - 10^7 k^2\ell(k+\ell)^2 tm \log(m) \stackrel{\eqref{eq: n' min deg}}{\geq} k. 
\end{align*}
Thus for every $u\in \widetilde{U}^{\rm sink}_i$, we can choose a set $\widehat{U}_{i,u} \subseteq N_{D'}^+(u)\setminus (U\cup C\cup \bigcup_{S\in \{A,B\}} (E^S_i\cup F^S)  \cup \bigcup_{j\in [i-1]} U^*_{j})$ with $|\widehat{U}_{i,u}|= k$. 
Let $\widehat{U}_i := \bigcup_{u \in \widetilde{U}^{\rm sink}_i} \widehat{U}_{i,u}$,  then we have $|\widehat{U}_i|\leq k(2k+\ell-2)$ and 
\begin{equation}\label{eq: Csharpi kconn 2}
\begin{minipage}[c]{0.9\textwidth}\em
for each $u\in  \widetilde{U}^{\rm sink}_i$, the pair
$(v, \widehat{U}_i\setminus (C\cup E^A_i\cup E^B_i\cup F^A\cup  F^B) )$ is $k$-connected in  $D[U_i\cup  \widetilde{U}_i\cup \widehat{U}_i]$.
\end{minipage}
\end{equation}
Let $U^*_i:= \widetilde{U}_i\cup \widehat{U}_i$, then $|U^*_i| \leq 2k(2k+\ell-2)\leq 4k(k+\ell)$, thus (U1)$^{i}$ follows, and $U^*_i$ is disjoint from the sets $U, U^*_1,\dots, U^*_{i-1}$ from its construction.

Now we prove that $U_i^*$ satisfies (U2)$^i$. 
Note that $U^*_i \cap (C\cup E^B_i\cup F^B) = \emptyset$.
Thus, for $v\in U^*_i$, it is clear that $(U^*_i\setminus (C\cup E^B_i\cup F^B),v )$ is $k$-connected in $D[U_i\cup U^*_i]$.
If $v\in U_i$, then Lemma~\ref{lem:glue} with \eqref{eq: Ci kconn} and \eqref{eq: Csharpi kconn} implies that $(U^*_i\setminus ( C\cup E^B_i\cup F^B),v )$ is $k$-connected in $D[U_i\cup U^*_i]$. 

Similarly, if $v\in \widehat{U}_i \subseteq U^*_i\setminus (C\cup E^A_i\cup E^B_i\cup F^A\cup F^B)$, then it is clear that $(v, U^*_i\setminus(C\cup E^A_i\cup E^B_i\cup F^A\cup F^B))$ is  $k$-connected in $D[U_i\cup U^*_i]$.
If $v\in \widetilde{U}^{\rm sink}_i$, then \eqref{eq: Csharpi kconn 2} implies that $(v, U^*_i\setminus(C\cup E^A_i\cup E^B_i\cup F^A\cup F^B))$ is $k$-connected in $D[U_i\cup U^*_i]$.
If $v\in  \widetilde{U}_i\cup U^{\rm sink}_i$, then Lemma~\ref{lem:glue}  with \eqref{eq: tilde Ci kconn} and \eqref{eq: Csharpi kconn 2} implies that $(v, U^*_i\setminus (C\cup E^A_i\cup E^B_i\cup F^A\cup F^B) )$ is $k$-connected in $D[U_i\cup U^*_i]$.
If $v\in U_i$, then Lemma~\ref{lem:glue}  with \eqref{eq: Ci kconn}
implies $(v, U^*_i\setminus (C\cup E^A_i\cup E^B_i\cup F^A\cup F^B) )$ is $k$-connected in $D[U_i\cup U^*_i]$. Thus $U^*_i$ satisfies (U2)$^{i}$.
By repeating this, we obtain desired pairwise disjoint sets $U^*_1,\dots, U^*_h$ satisfying (U1)$^{h'}$ and (U2)$^{h'}$, thus \ref{U*1} and \ref{U*2}. This proves the claim.
\end{proof}
 \vspace{0.3cm}

\noindent {\bf Step 3. Construction of sets $W'_1,\dots, W'_t$.} 
In this step, we construct sets $W'_1,\dots, W'_{t}$ satisfying the following \ref{W'0}--\ref{W'3}. Each $W'_i$ together with some more vertices will give the desired set $W_i$ later.

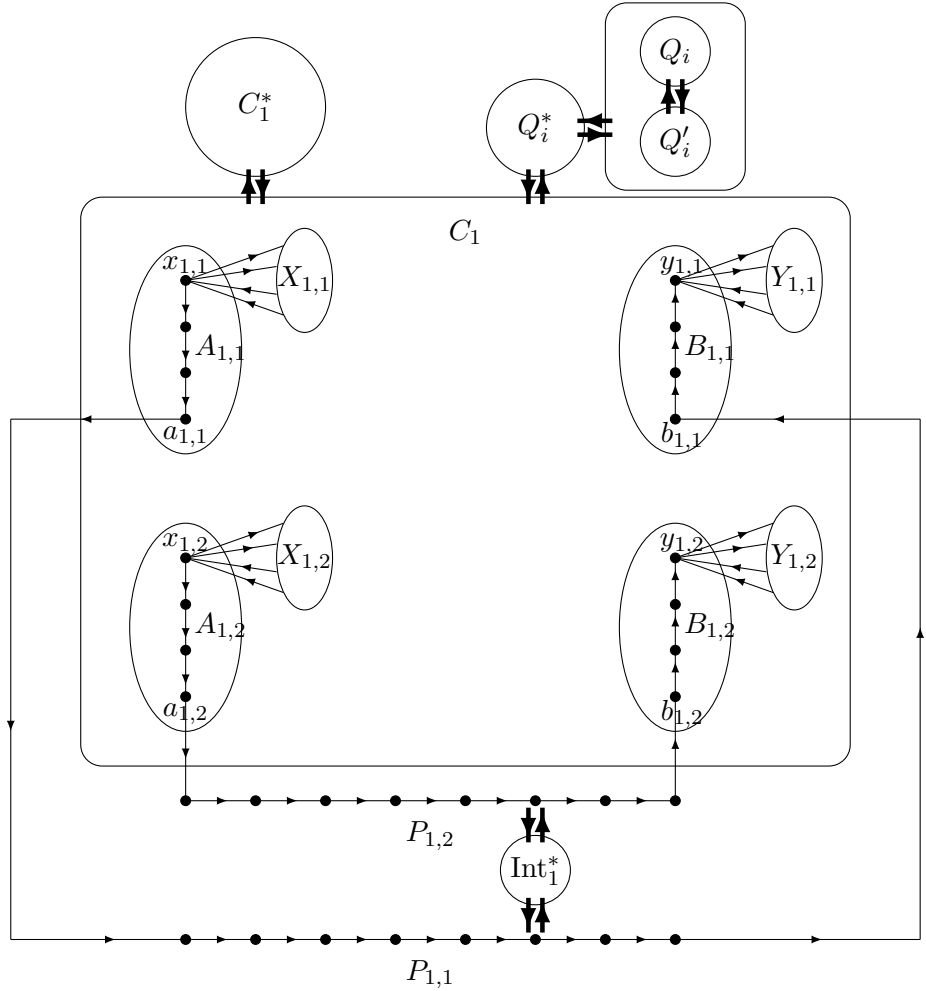
\begin{figure}
\centering
\begin{tikzpicture}[scale=0.92]

  \draw[fill=none, rounded corners=8] (-0.5,-6) rectangle (10.5,2.2);

\node at (5,1.7) {$C_1$};

\draw[middleupuparrow={latex}, line width=0.6mm] (1.9,2.1) -- (1.9,2.1+0.5) ;
\draw[middleupuparrow={latex}, line width=0.6mm] (2.1,2.1+0.5) -- (2.1,2.1) ;

  \draw[fill=none] (2,3.5) ellipse [x radius=1,y radius=1];
  \node at (2,3.5) {$C_1^*$};

\draw[middleupuparrow={latex}, line width=0.6mm] (5.9,2.1+0.5) -- (5.9,2.1) ;
\draw[middleupuparrow={latex}, line width=0.6mm] (6.1,2.1) -- (6.1,2.1+0.5) ;

  \draw[fill=none] (6,2.7+0.5) ellipse [x radius=0.7,y radius=0.7];
  \node at (6,2.7+0.5) {$Q_i^*$};

\draw[middleupuparrow={latex}, line width=0.6mm] (7.1,3.3) -- (6.6,3.3) ;
\draw[middleupuparrow={latex}, line width=0.6mm] (6.6,3.1) -- (7.1,3.1) ;

  \draw[fill=none, rounded corners=8] (7,2.3) rectangle (9,5);

  \draw[fill=none] (8,3) ellipse [x radius=0.5,y radius=0.5];
  \node at (8,3) {$Q'_i$};

  \draw[fill=none] (8,4.3) ellipse [x radius=0.5,y radius=0.5];
  \node at (8,4.3) {$Q_i$};

\draw[middleupuparrow={latex}, line width=0.6mm] (7.9,3.4) -- (7.9,3.9) ;
\draw[middleupuparrow={latex}, line width=0.6mm] (8.1,3.9) -- (8.1,3.4) ;

  \draw[fill=none] (6,-7.5) ellipse [x radius=0.5,y radius=0.5];
  \node at (6,-7.5) {$\Int_1^*$};

\draw[middleupuparrow={latex}, line width=0.6mm] (6.1,-7.1) -- (6.1,-6.6) ;
\draw[middleupuparrow={latex}, line width=0.6mm] (5.9,-6.6) -- (5.9,-7.1);
  
\draw[middleupuparrow={latex}, line width=0.6mm] (5.9,-7.9) -- (5.9,-8.4) ;
\draw[middleupuparrow={latex}, line width=0.6mm] (6.1,-8.4) -- (6.1,-7.9) ;

  \draw[fill=none] (8,0) ellipse [x radius=0.8,y radius=1.5];

\node at (8.5,0) {$B_{1,1}$};

  \draw[fill=none] (8,-4) ellipse [x radius=0.8,y radius=1.5];

\node at (8.5,-4) {$B_{1,2}$};

 \draw[fill=none] (1,0) ellipse [x radius=0.8,y radius=1.5];
\node at (1.5,0) {$A_{1,1}$};
 \draw[fill=none] (1,-4) ellipse [x radius=0.8,y radius=1.5];
\node at (1.5,-4) {$A_{1,2}$};

\filldraw[fill=black] (1,1) circle (2pt);
\filldraw[fill=black] (1,0.33) circle (2pt);
\filldraw[fill=black] (1,-0.33) circle (2pt);
\filldraw[fill=black] (1,-1) circle (2pt);

\draw[middleuparrow={latex}] (1,1) -- (1,0.33) ;
\draw[middleuparrow={latex}] (1,0.33) -- (1,-0.33) ;
\draw[middleuparrow={latex}] (1,-0.33) -- (1,-1) ;

\node at (1,-1.25) {$a_{1,1}$};
\node at (1,1.2) {$x_{1,1}$};

 \draw[fill=none] (2.7,1) ellipse [x radius=0.4,y radius=0.75];
 \node at (2.7,1) {$X_{1,1}$};
\draw[middleuparrow={latex}] (1,1) -- (2.4,1.5);
\draw[middleuparrow={latex}] (1,1) -- (2.3,1.2);
\draw[middledownarrow={latex}] (2.3,0.8) -- (1,1);
\draw[middledownarrow={latex}] (2.4,0.5) -- (1,1);

\filldraw[fill=black] (1,-3) circle (2pt);
\filldraw[fill=black] (1,-3.67) circle (2pt);
\filldraw[fill=black] (1,-4.33) circle (2pt);
\filldraw[fill=black] (1,-5) circle (2pt);

\draw[middleuparrow={latex}] (1,-3) -- (1,-3.67) ;
\draw[middleuparrow={latex}] (1,-3.67) -- (1,-4.33) ;
\draw[middleuparrow={latex}] (1,-4.33) -- (1,-5) ;

\node at (1,-5.25) {$a_{1,2}$};
\node at (1,-2.8) {$x_{1,2}$};

 \draw[fill=none] (2.7,-3) ellipse [x radius=0.4,y radius=0.75];
 \node at (2.7,-3) {$X_{1,2}$};
\draw[middleuparrow={latex}] (1,-3) -- (2.4,-2.5);
\draw[middleuparrow={latex}] (1,-3) -- (2.3,-2.8);
\draw[middledownarrow={latex}] (2.3,-3.2) -- (1,-3);
\draw[middledownarrow={latex}] (2.4,-3.5) -- (1,-3);

\filldraw[fill=black] (8,1) circle (2pt);
\filldraw[fill=black] (8,0.33) circle (2pt);
\filldraw[fill=black] (8,-0.33) circle (2pt);
\filldraw[fill=black] (8,-1) circle (2pt);

\draw[middleuparrow={latex}] (8,0.33) -- (8,1);
\draw[middleuparrow={latex}] (8,-0.33) -- (8,0.33) ;
\draw[middleuparrow={latex}] (8,-1) -- (8,-0.33) ;

\node at (8.1,-1.25) {$b_{1,1}$};
\node at (8.1,1.2) {$y_{1,1}$};

 \draw[fill=none] (9.7,1) ellipse [x radius=0.4,y radius=0.75];
 \node at (9.7,1) {$Y_{1,1}$};
\draw[middleuparrow={latex}] (8,1) -- (9.4,1.5);
\draw[middleuparrow={latex}] (8,1) -- (9.3,1.2);
\draw[middledownarrow={latex}] (9.3,0.8) -- (8,1);
\draw[middledownarrow={latex}] (9.4,0.5) -- (8,1);

\filldraw[fill=black] (8,-3) circle (2pt);
\filldraw[fill=black] (8,-3.67) circle (2pt);
\filldraw[fill=black] (8,-4.33) circle (2pt);
\filldraw[fill=black] (8,-5) circle (2pt);

\draw[middleuparrow={latex}] (8,-3.67) -- (8,-3);
\draw[middleuparrow={latex}] (8,-4.33) -- (8,-3.67) ;
\draw[middleuparrow={latex}] (8,-5) -- (8,-4.33) ;

\node at (8.1,-5.25) {$b_{1,2}$};
\node at (8.1,-2.8) {$y_{1,2}$};

 \draw[fill=none] (9.7,-3) ellipse [x radius=0.4,y radius=0.75];
 \node at (9.7,-3) {$Y_{1,2}$};
\draw[middleuparrow={latex}] (8,-3) -- (9.4,-2.5);
\draw[middleuparrow={latex}] (8,-3) -- (9.3,-2.8);
\draw[middledownarrow={latex}] (9.3,-3.2) -- (8,-3);
\draw[middledownarrow={latex}] (9.4,-3.5) -- (8,-3);

\filldraw[fill=black] (1,-8.5) circle (2pt);
\filldraw[fill=black] (2,-8.5) circle (2pt);
\filldraw[fill=black] (3,-8.5) circle (2pt);
\filldraw[fill=black] (4,-8.5) circle (2pt);
\filldraw[fill=black] (5,-8.5) circle (2pt);
\filldraw[fill=black] (6,-8.5) circle (2pt);
\filldraw[fill=black] (7,-8.5) circle (2pt);
\filldraw[fill=black] (8,-8.5) circle (2pt);

\draw[middlearrow={latex}] (1,-1) -- (-1.5,-1);
\draw[middlearrow={latex}] (-1.5,-1) -- (-1.5,-8.5);

\draw[middlearrow={latex}] (-1.5,-8.5) -- (1,-8.5);
\draw[middlearrow={latex}] (1,-8.5) -- (2,-8.5);
\draw[middlearrow={latex}] (2,-8.5) -- (3,-8.5);
\draw[middlearrow={latex}] (3,-8.5) -- (4,-8.5);
\draw[middlearrow={latex}] (4,-8.5) -- (5,-8.5);
\draw[middlearrow={latex}] (5,-8.5) -- (6,-8.5);
\draw[middlearrow={latex}] (6,-8.5) -- (7,-8.5);

\draw[middlearrow={latex}] (7,-8.5) -- (8,-8.5);

\draw[middlearrow={latex}] (8,-8.5) -- (11.5,-8.5);

\draw[middlearrow={latex}] (11.5,-1) -- (8,-1);
\draw[middlearrow={latex}] (11.5,-8.5) -- (11.5,-1);

\node at (4.5,-7) {{\bf $P_{1,2}$}};

\node at (4.5,-9) {{\bf $P_{1,1}$}};

\filldraw[fill=black] (1,-6.5) circle (2pt);
\filldraw[fill=black] (2,-6.5) circle (2pt);
\filldraw[fill=black] (3,-6.5) circle (2pt);
\filldraw[fill=black] (4,-6.5) circle (2pt);
\filldraw[fill=black] (5,-6.5) circle (2pt);
\filldraw[fill=black] (6,-6.5) circle (2pt);
\filldraw[fill=black] (7,-6.5) circle (2pt);
\filldraw[fill=black] (8,-6.5) circle (2pt);

\draw[middlearrow={latex}] (1,-5) -- (1,-6.5);
\draw[middlearrow={latex}] (1,-6.5) -- (2,-6.5);
\draw[middlearrow={latex}] (2,-6.5) -- (3,-6.5);
\draw[middlearrow={latex}] (3,-6.5) -- (4,-6.5);
\draw[middlearrow={latex}] (4,-6.5) -- (5,-6.5);
\draw[middlearrow={latex}] (5,-6.5) -- (6,-6.5);
\draw[middlearrow={latex}] (6,-6.5) -- (7,-6.5);
\draw[middlearrow={latex}] (7,-6.5) -- (8,-6.5);
\draw[middlearrow={latex}] (8,-6.5) -- (8,-5);

\end{tikzpicture}
\caption{Structure of $Z_1$.}\label{fig: Z}
\end{figure}

\begin{PROP}\label{prop: W'}
After permuting the indices in $[h]$ (hence the sets $C_1,\dots, C_{h}$), there exist pairwise disjoint sets $W'_1,\dots, W'_{t}\subseteq V$ of vertices and pairwise vertex-disjoint paths 
$P_{1,1},\dots, P_{t,3k}$ that satisfy the following for each $i\in [t]$, where $W'=\bigcup_{i\in [t]} W'_i$.
\begin{enumerate}[label=\text{\rm (W$'$\arabic*)}]
\item\label{W'0} $Q_i \cup C_{i} \cup \bigcup_{j\in [3k]} \Int(P_{i,j}) \subseteq W'_i$ and $C\subseteq W'$.
\item \label{W'1}$3k +\ell \leq |W'_i| \leq \frac{n}{180mt}$.
\item \label{W'2} For each $v\in W'_i$, both $(v, V'\setminus(C\cup E^A_i))$ and $(V'\setminus (C \cup E^B_i),v )$ are $k$-connected in $D[W'_i]$.
\item \label{W'3}For each $j\in [3k]$, the path $P_{i,j}$ starts at $a_{i,j}$ and ends at $b_{i,j}$.
\end{enumerate}
\end{PROP}
\begin{proof}

To prove this proposition, we will first construct pairwise disjoint vertex sets $Z_1,\dots, Z_t$ which satisfy \ref{W'2}, \ref{W'3} and the first part of \ref{W'0}. For this, we will construct sets $C_i, C^*_i, Q^{*}_i, \Int^*_i$ and paths $P_{i,j}$ for all $i\in [h]$ and $j\in [3k]$. After that, we will discard some of these sets and we will permute the indices in $[t]$ so that the remaining sets are $Z_1,\dots, Z_t$ containing $C_1,\dots, C_t$, respectively.  Once we obtain such sets, we will distribute more vertices into appropriate sets $Z_i$ to obtain a desired set $W'_i$ that also satisfies~\ref{W'0} and~\ref{W'1}.  Figure~\ref{fig: Z} shows how we will construct $Z_i$.

First, we construct disjoint sets $C^*_1,\dots, C^*_h$. The below property \ref{C*2} ensures that every exceptional vertex in $C_i\cup C^*_i$ can reach to $V'\setminus E^A_i$ and can be reached from $V'\setminus E^B_i$ even after deleting at most $k-1$ vertices.
Moreover, as we plan to later discard $C_i\cup C^*_i$ for some $i\in [h]$, we wish to be able to reassign them to $Z_j$ for many appropriate choices of $j\in [h]$ to satisfy $C \subseteq W'$, without making each set $Z_j$ too large. Note that \ref{C*2} also ensures that every vertex can reach to $V'\setminus F^A$ and can be reached from $V'\setminus F^B$ even after deleting at most $k-1$ vertices, and the vertices $v$ outside $F^A$ or $F^B$ have many `good' choices of $Z_{j}$ to ensure~\ref{W'2}, allowing us to reassign the vertices in $C_i\cup C^*_i$ together to some $Z_j$.

\begin{CLAIM}\label{claim:c*}
There exist pairwise disjoint sets $C^*_1,\dots, C^*_h \subseteq V'\setminus (Q' \cup C)$ satisfying the following. For all $i\in [h]$ and $v\in C_i\cup C^*_i$,
\begin{enumerate}[label=\text{\rm (C$^*$\arabic*)}]
\item \label{C*0} The sets $C , C^*_1 , \dots , C^*_h , Q_1 , \dots , Q_t , Q_1' , \dots , Q_t'$ are pairwise disjoint.
\item\label{C*1} $|C^*_i| \leq 4k(k+\ell)$.
\item\label{C*2} Both $(v,C^*_i\setminus (E^A_i\cup F^A) )$ and $(C^*_i\setminus (E^B_i\cup F^B),v )$ are $k$-connected in $D[C_i\cup C^*_i]$.
\end{enumerate}
\end{CLAIM}
\begin{proof}
We apply Claim~\ref{cl: extension} with $C_i\setminus V_{\rm low}$ and $Q' \cup C$ playing the roles of $U_i$ and $U$, respectively, which is possible by~\eqref{eq: size C} and~\eqref{eqn:q'}. Then we obtain sets $C^*_1,\dots, C^*_h$ that satisfy~\ref{C*0} and \ref{C*1}, and \ref{C*2} holds for all vertices $v\in (C_i\cup C^*_i)\setminus V_{\rm low}$.

For $u\in C_i\cap V_{\rm low}$, it follows that $(v,C_i\setminus V_{\rm low})$ and $(C_i\setminus V_{\rm low},v)$ are both $k$-connected in $D[C_i]$ by~\ref{XY3}. Hence Lemma~\ref{lem:glue} implies~\ref{C*2}, which proves the claim.
\end{proof}
Let $C^*:=\bigcup_{i\in [h]} C^*_i$. By~\eqref{eq: size C} with~\ref{C*1}, we have
\begin{align}\label{eq: size C C*}
|C\cup C^*| \leq 5 \cdot 10^5 k^2(k+ \ell)^2 tm\log(m).
\end{align}

Using the similar argument as the proof of Claim~\ref{claim:c*}, we can prepare sets $Q^*_i$ as the following claim. These sets will ensure that for each $i\in [t]$, there are many choices of $j$ such that we can assign vertices in $Q_i\cup Q'_i\cup Q^*_i$ together into $W_j$.
\begin{CLAIM}\label{claim:q*}
There exist pairwise disjoint sets $Q^*_1,\dots, Q^*_t \subseteq V'\setminus (Q \cup Q' \cup C\cup C^*)$ satisfying the following. For all $i\in [t]$ and $v\in Q_i\cup Q_i' \cup Q^*_i$,
\begin{enumerate}[label=\text{\rm (Q$^*$\arabic*)}]
\item\label{Q*0} The sets $C , C^* , Q_1^{\#} , \dots , Q_t^{\#}$ are pairwise disjoint, where $Q_i^{\#} := Q_i \cup Q_i' \cup Q_i^*$ for each $i \in [t]$.
\item\label{Q*1} $|Q^*_i| \leq 4k(k+\ell)$.
\item\label{Q*2} Both $(v,Q^*_i\setminus F^A )$ and $(Q^*_i\setminus F^B,v )$ are $k$-connected in $D[Q_i\cup Q_i' \cup Q^*_i]$.
\end{enumerate}
\end{CLAIM}
\begin{proof}
By~\eqref{eqn:q'} and \eqref{eq: size C C*}, for each $i \in [t]$ we can apply Claim~\ref{cl: extension} with $Q_i'$ and $Q \cup Q' \cup C\cup C^*$ playing the roles of $U_i$ and $U$ respectively. This gives $t$ disjoint sets $Q_1^* , \dots , Q_t^*$ satisfying \ref{Q*0} and \ref{Q*1}. Moreover, by \ref{Q'2} and \ref{U*2}, we can deduce that $Q_1^* , \dots , Q_t^*$ satisfy \ref{Q*2}.
\end{proof}

Recall that $Q^{\#}_i = Q_i\cup Q'_i \cup Q^*_i$ for each $i \in [t]$. Let $Q^{\#} := \bigcup_{i=1}^{t} Q^{\#}_i$. 
By~\eqref{eqn:q'} and~\ref{Q*1}, for each $i\in [t]$, we have
\begin{align}\label{eq: size Q Q' Q*}
|Q^{\#}_i| \leq q+ 8k(k+\ell) \text{ and } |Q^{\#}| \leq qt + 8k(k+\ell)t.
\end{align}

Now we will construct paths $P_{i,j}$ from $a_{i,j}$ to $b_{i,j}$.
In order to satisfy~\ref{A1}, we ensure that the length of each $P_{i,j}$ is not too long; the factor of $m$ in the definition of $h$ is for this. After constructing $3kh$ paths, we will choose short ones and discard the rest of them. However, discarding some paths and dominating sets might cause some issues which we will handle later.

Since $D'$ is $9 \cdot 10^7 q k^2 \ell(k+\ell)^2 tm^2\log(m)$-connected and
$9 \cdot 10^{7} q k^2 \ell(k+\ell)^2tm^2\log(m) - |Q^{\#} \cup C \cup C^*| \geq 452\cdot 3 k h+ 188\ell$, the subgraph 
$D' \setminus ((Q^{\#} \cup C\cup C^*)\setminus \{a_{i,j},b_{i,j} : (i,j)\in [h]\times [3k]\})$ is strongly $(452\cdot 3k h + 188 \ell)$-connected, and thus $3kh$-linked by  Corollary~\ref{cor:digraph_link}.

Hence there is a collection $\{P_{i,j}: (i,j)\in [h]\times [3k]\}$ of vertex-disjoint paths in $D' \setminus ((Q^{\#} \cup C\cup C^*)\setminus \{a_{i,j},b_{i,j} : (i,j)\in [h]\times [3k]\})$ such that 
\begin{enumerate}[label=\text{\rm (P\arabic*)}]
\item\label{P1} $P_{i,j}$ is a path from $a_{i,j}$ to $b_{i,j}$, and
\item\label{P2} The sets $C , C^* , Q^{\#} , \Int_1 , \dots , \Int_h$ are pairwise disjoint, where $\Int_{i}:=\bigcup_{j=1}^{3k} \Int(P_{i,j})$ for each $i\in [h]$.
\end{enumerate}
By permuting indices on $[h]$ if necessary, we may assume that for $i\in [h/6]$ and $i' \in [h]\setminus [h/6]$, we have
$|C_{i}\cup C^*_{i}\cup \Int_{i}| \leq \ |C_{i'}\cup C^*_{i'}\cup \Int_{i'}|$.

\begin{CLAIM}\label{cl: path short}
For each $i\in [h/6]$, we have
$|C_i\cup C^*_i\cup \Int_i|\leq \frac{n'}{1000mt }$ 
\end{CLAIM}
\begin{proof}
For $i\neq j\in [h]$, the set $C_i\cup C^*_i\cup \Int_i$ is disjoint from $C_j\cup C^*_j\cup \Int_j$.
If the claim is not true, then we have 
$$|V'| \geq \sum_{i\in [h]\setminus [h/6]} |C_i\cup C^*_i\cup \Int_i|
> \frac{5 h n'}{6000mt } \geq n',$$
a contradiction. This proves the claim.
\end{proof}
As some $P_{i,j}$ may contain exceptional vertices, we need to `take care of' such vertices. We will construct small sets $\Int^*_i$ so that $\Int_i$ can be assigned to $W'_i$ as long as we assign it together with $\Int_i^*$.


\begin{CLAIM}\label{cl: Int*}
There exist pairwise disjoint sets $\Int^*_1,\dots, \Int^*_{h/6} \subseteq V'\setminus (Q^{\#} \cup C\cup C^*)$ such that the following holds for all $i\in [h/6]$ and $u\in \Int_i\cup \Int^*_i$.
\begin{enumerate}[label=\text{\rm (Int$^*$\arabic*)}]
\item \label{Int*0} The sets $C , C^* , Q^{\#} , \Int^*_1 , \dots , \Int^*_{h/6}$ are pairwise disjoint.
\item\label{Int*1} $|\Int^*_i| \leq 4k(k+\ell)$.
\item\label{Int*2} Both $(u,\Int^*_i\setminus (C\cup E^A_i\cup F^A) )$ and $(\Int^*_i \setminus (C\cup E^B_i\cup F^B),u )$ are $k$-connected in $D[\Int_i\cup \Int^*_i]$.
\end{enumerate}
\end{CLAIM}
\begin{proof}
For each $i\in [h/6]$, we have $\Int_i\cap V_{\rm low} \subseteq \Int_i\cap C = \emptyset$. By this with \eqref{eq: size C C*} and \eqref{eq: size Q Q' Q*}, we can apply Claim~\ref{cl: extension} with $\Int_i$ and $Q^{\#} \cup C\cup C^*$ playing the role of $U_i$ and $U$, respectively.
Then we obtain sets $\Int^*_1,\dots, \Int^*_{h/6}$ satisfying all \ref{Int*0}, \ref{Int*1} and \ref{Int*2}.
\end{proof}
Note that the condition~\ref{Int*1} holds no matter how big $|V(P_{i,j})|$ is.
For each $i\in [h/6]$, let 
$$Z'_i:= C_i\cup C^*_{i} \cup \Int_i \cup \Int^*_{i}.$$
 Although the set $\Int^*_i$ takes care of the exceptional vertices in $\Int_i$ by~\ref{Int*2}, the set $\Int^*_i$  may intersect with another path $P_{i',j'}$, hence $\{ Z'_i \}_{i \in [h/6]}$ does not necessarily consist of pairwise disjoint sets. We utilise the following Claim~\ref{lem:indep} to keep only pairwise disjoint sets $Z'_i$.

\begin{CLAIM}\label{lem:indep}
There is a set $I \subseteq [h/6]$ with $|I|=t$ satisfying the following.
\begin{enumerate}[label=\text{\rm (I\arabic*)}]
\item\label{I1} For all $i\neq j \in I$, we have
$Z'_i \cap Z'_{j}=\emptyset$.

\item\label{I2} There is a bijective function $\phi : [t] \to I$ such that for each $i \in [t]$ and all $v \in Q_i^* \setminus F^A$ and $u \in Q_i^* \setminus F^B$, we have $v \notin E_{\phi(i)}^A$ and $u \notin E_{\phi(i)}^B$. 
\end{enumerate}
\end{CLAIM}
\begin{proof}
Let $D^{\rm aux}$ be a digraph on a vertex-set $[h/6]$ such that 
$\overrightarrow{ij} \in E(D^{\rm aux})$ if $\Int^*_i\cap \Int_j\neq \emptyset$ and
$G^{\rm aux}$ be a graph on a vertex-set $[h/6]$ obtained from $D^{\rm aux}$ by removing orientation of the edges and by removing parallel edges. 
By \ref{P2} and \ref{Int*0}, $ij \notin E(G^{\rm aux})$ implies $Z'_i\cap Z'_j=\emptyset$.

As $\Int_1,\dots, \Int_{h/6}$ are pairwise disjoint, it is easy to see that for each $i\in [h/6]$, we have $d^+_{D^{\rm aux}}(i) \leq |\Int^*_i|$.
Thus 
$$|E(G^{\rm aux})| \leq |E(D^{\rm aux})| \leq \sum_{i\in [h/6]} d^+_{D^{\rm aux}}(i) \stackrel{\ref{Int*1}}{\leq } \frac{4k(k+\ell)h}{6}.$$

By Tur\'an's theorem, any $n$-vertex graph with average degree $d$ has an independent set of size at lease $n/(d+1)$. Hence there exists an independent set $I_0 \subseteq [h/6]$ of $G^{\rm aux}$ with size $2t \leq \frac{h}{6\cdot 8k(k+\ell)+1}$. 

Now it is easy to see that every subset of $I_0$ satisfies~\ref{I1} as $I$ is an independent set in $G^{\rm aux}$, and thus it is enough to show that there is a subset $I \subseteq I_0$ with $|I|=t$ and $\phi : [t]\rightarrow I$ satisfying~\ref{I2}. For each $i \in [t]$, since 
\begin{eqnarray*}
\sum_{S\in \{A,B\}}\Big| \hspace{-0.1cm} \bigcup_{v\in Q^*_i \setminus F^S} \{ j \in [t] : v \in E^{S}_j \}\Big|
&\stackrel{\eqref{eq: F def}}{\leq} & \hspace{-0.1cm} 
\big(|Q^*_i \setminus F^A| 
+ |Q^*_i \setminus F^B| \big) \frac{t}{100k(k+\ell)} 
\stackrel{\ref{Q*1}}{\leq } \hspace{-0.1cm}  \frac{4t}{25},
\end{eqnarray*}
there exists a set $J_i \subseteq [2t]$ with $|J_i| \geq 2t - \frac{4t}{25} > t$, where for all $v \in Q_i^* \setminus F^A$ and $u \in Q_i^* \setminus F^B$, we have $v \notin E_j^A$ and $u \notin E_j^B$ for each $j \in J_i$. Hence we can greedily choose $\phi(i) \in J_i$ for each $i \in [t]$ so that $\phi$ is injective. Let us define $I := \phi([t])$, which satisfies~\ref{I2}. This proves the claim.
\end{proof}

By permuting indices on $[h/6]$, we assume that $\phi(i) = i$ for each $i \in [t]$, and thus $I = [t]$. Note that, by \ref{Int*0}, for all $i\in [h]\setminus [t]$ and $j\in [t]$, the vertices in $C_i\cup C^*_i$ do not intersect with $Z'_j$. 
For each $i\in [t]$, let 
$$Z_i:= Z'_i\cup Q^{\#}_i,$$
then $Z_1,\dots, Z_t$ are pairwise disjoint and Claim~\ref{cl: path short},\ref{Int*1} and \eqref{eq: size Q Q' Q*} imply that,
for each $i\in [t]$,
\begin{align}\label{eq: Z size}
|Z_i| \leq \frac{n'}{500mt}.
\end{align}

Now we want to add some more vertices to each $Z_i$ to get the desired set $W_i$. However, some discarded vertices in $C_i\cup C^*_i$ for $i\in [h]\setminus[t]$ might be dangerous to guarantee~\ref{A3}, because~\eqref{eq: E def} and~\eqref{eq: F def} only consider vertices in $V' \setminus C$ and thus we have no information of the vertices in $C$ on whether they are in/out-dominated by almost dominating sets.

Therefore, we need to assign the vertices in these sets to one of $Z_j$ for $j\in [t]$. This issue can be handled by using the property \ref{C*2} as follows. Note that if $W'_i$ is too small, then it is difficult to control the non-edges between two vertices. Hence we will add a set $W''_i$ of a few vertices into $W'_i$ to obtain the lower bound in \ref{W'1}.

\begin{CLAIM}\label{cl: J}
For each $a\in [h]\setminus [t]$, we can choose $j(a)\in [t]$ satisfying the following. 
\begin{enumerate}[label=\text{\rm (j\arabic*)}]
\item\label{j1} $C^*_a \setminus F^A$ and $E^{A}_{j(a)}$ are disjoint, and
$C^*_a\setminus F^B$ and $ E^{B}_{j(a)}$ are disjoint.

\item\label{j2} For each $i \in [t]$, we have $|J_i|\leq \lceil\frac{h}{23t/25}\rceil \leq \frac{ 6h}{5t}$, where $J_{i}:= \{ i'\in [h]\setminus [t]: j(i')= i \}.$
\end{enumerate}
\end{CLAIM}
\begin{proof}
For each $a \in [h]\setminus [t]$, 
we have
\begin{eqnarray*}
\sum_{S\in \{A,B\}}\Big| \hspace{-0.1cm} \bigcup_{v\in C^*_{a} \setminus F^S} \{ i \in [t] : v \in E^{S}_i \}\Big|
&\stackrel{\eqref{eq: F def}}{\leq} & \hspace{-0.1cm} 
\big(|C^*_{a} \setminus F^A| 
+ |C^*_{a} \setminus F^B| \big) \frac{t}{100k(k+\ell)} 
\stackrel{\ref{C*1}}{\leq } \hspace{-0.1cm}  \frac{2t}{25}.
\end{eqnarray*}
Thus, for every $a\in [h]\setminus [t]$, there are at least $\frac{23t}{25}$ indices $i\in [t]$ such that $C^*_{a}\setminus F^S$ and $E^S_{i}$ are disjoint for each $S\in \{A,B\}$. For each $a\in [t]\setminus [t]$, we can choose $j(a) \in [t]$  among those $\frac{23t}{25}$ choices in a way that \ref{j2} holds. Moreover, this choice guarantees \ref{j1}.
\end{proof}

By Claim~\ref{cl: path short} and \ref{Int*1}, 
for every $i\in [t]$,
$$ \Big|V\setminus \Big(Q^{\#} \cup E^A_i\cup E^B_i\cup \bigcup_{i'\in [t]} (\Int_{i'}\cup \Int^*_{i'})\cup (C\cup C^*)\Big) \Big| \stackrel{\eqref{eq: size Q Q' Q*},\eqref{eq: E size},\eqref{eq: n' min deg}}{\geq } t\ell.$$
Therefore, for every $i\in [t]$, we may choose a set $W''_i$ of $\ell$ vertices in 
$V\setminus \Big(Q^{\#} \cup E^A_i\cup E^B_i\cup \bigcup_{i'\in [t]} (\Int_{i'}\cup \Int^*_{i'})\cup (C\cup C^*)\Big)$
such that $W''_i\cap W''_{j}=\emptyset$ for all $i\neq j\in [t]$. For every $i\in [t]$, let
$$W'_i := W''_i\cup Z_i \cup \bigcup_{i'\in J_i} (C_{i'}\cup C^*_{i'}).$$

Now we aim to show $W'_i$ satisfies~\ref{W'0}--\ref{W'3} for each $i \in [t]$.
First, \ref{W'0} is obvious from the construction and the fact that $J_1,\dots, J_t$ partitions $[h]\setminus [t]$.
Also \ref{P2} and~\ref{I1} imply that $W'_1,\dots, W'_t$ are pairwise disjoint and \ref{P2} with the definition of $P_{i,j}$ imply that $P_{1,1},\dots, P_{t,3k}$ are pairwise disjoint paths satisfying \ref{W'3}.

Now we show~\ref{W'2}. For each $v\in Q_i^*\setminus F^A$, since $v\notin C$, \ref{I2} implies that 
$v \in V'\setminus (C\cup E^A_i)$, thus $(v, V'\setminus (C\cup E^A_i))$ is $k$-connected in $D[W'_i]$.
For each $v \in Q^{\#}_i$, \ref{Q*2} with Lemma~\ref{lem:glue} implies that $(v,V'\setminus (C\cup E_i^A))$ is $k$-connected in $D[W'_i]$. 

Similarly, for all $i'\in J_i$ and $v\in C^*_{i'} \setminus F^A$, since $v\notin C$,  \ref{j1} implies that $v\in V'\setminus(C\cup E_i^A)$, thus $(v, V'\setminus(C\cup E_i^A))$ is $k$-connected in $D[W'_i]$. Since this holds for all $i'\in J_i$ and $v\in C^*_{i'}\setminus F^A$,
Lemma~\ref{lem:glue} together with~\ref{C*2} and \ref{Int*2} implies that for any $v\in W'_i\setminus W''_i$, the ordered pair $(v, V'\setminus(C\cup E^A_i))$ is $k$-connected in $D[W'_i]$. For each $v\in W''_i$, since $v \in V' \setminus (C \cup E_i^A)$, it is clear that $(v,V'\setminus(C\cup E^A_i))$ is $k$-connected in $D[W'_i]$.
Similarly, we can show that for any $v\in W'_i$, the ordered pair $(V'\setminus(C\cup E^B_i) ,v)$ is $k$-connected in $D[W'_i]$, hence \ref{W'2} holds.

Finally, we verify~\ref{W'1}. For every $i\in [t]$, we have
\begin{eqnarray}\label{eq: W'i size}
|W'_i| &\leq& \sum_{i'\in J_i} |C_{i'}\cup C^*_{i'}| + |Z_i| + |W''_i| \nonumber \\
&\stackrel{\substack{\ref{AB1},\ref{C*1}} }{\leq}& |J_i|\big(250k^2 + 30k \log(km\ell) + 4k(k+\ell)\big) + |Z_i| +  \ell \nonumber \\
&\stackrel{\ref{j2},\eqref{eq: Z size}}{\leq}& \frac{h(310k(k+\ell) + 36k\log(km\ell) )}{t} + \frac{n}{500mt} + \ell \stackrel{\eqref{eq: n min deg}}{\leq} \frac{n}{180mt},
\end{eqnarray}\COMMENT{The last place that we need $m^2\log{m}$-term in the connectivity.
We can instead assume $n\geq (...)m^2\log{m}$ while connectivity is at least 
$(...)m\log{m}$-connected.}
where $3k + \ell \leq |W'_i|$ comes from \eqref{eq: n' min deg} and the fact that $|W''_i|=\ell$, hence \ref{W'1} holds. This completes the proof of the proposition.

\end{proof}

Using Proposition~\ref{prop: W'}, we re-order the indices in $[h]$ and 
find sets $W'_1,\dots, W'_t$ and paths $P_{1,1},\dots, P_{t,3k}$ satisfying \ref{W'0}--\ref{W'3}. Let
$W':= \bigcup_{i\in [t]} W'_i$, then we have $|W'| \stackrel{\ref{W'1}}{\leq} \frac{n}{180m}.$
\vspace{0.3cm}

\noindent {\bf Step 4. Distribution of updated exceptional vertices.}

In this step, we aim to distribute exceptional vertices not satisfying \ref{A3} into $W'_1 , \dots , W'_t$, without increasing each size not too much. The resulting pairwise disjoint sets will be the final sets $W_1 , \dots , W_t$ of Lemma~\ref{lem:main}.

From the construction, $W'_i$ may contain many vertices other than the vertices in $A_{i,j}\cup B_{i,j}\cup V(P_{i,j})$;  some exceptional vertices in $E^{A}_i\cup E^{B}_i$ may have many in-neighbors and out-neighbors in $W'_i$, and thus become non-exceptional at this stage. Moreover, note that we defined $F^{A}, F^{B}$  based on the partition over $[h]$ and as we discard many parts from the linkage structures and only preserve linkage structures $A_{i,j}\cup B_{i,j}\cup P_{i,j}\subseteq W'_i$ for $i\in [t]$, this changes the meaning of `exceptional'. Hence we consider the ``updated'' sets $\widetilde{E}^A_i, \widetilde{E}^B_i, \widetilde{F}^A$ and $\widetilde{F}^B$ of exceptional vertices. Specific structures inside $W_i$ will not be important any more in this step.

One of the main difficulties is that for a vertex $v\in \widetilde{F}^B$, its in-degree provided by \ref{AB3} can be much smaller than $|\widetilde{F}^A\cup \widetilde{F}^B|$, as $\delta^-_0$ may be significantly smaller than $\delta^+_0$; in this case, we may not be able to find non-exceptional in-neighbors of $v$ which we will assign together with $v$. To handle this problem, we first distribute vertices in $\widetilde{F}^A$, and then we only define $\widetilde{F}^B$ after the distribution of $\widetilde{F}^A$. \newline 

\noindent {\bf Step 4.1. Distribution of vertices in $\widetilde{F}^A$.}
We define updated sets of exceptional vertices. For every  $i\in [t]$, let 
\begin{align}
\begin{split}\label{eq: tilde E def}
\widetilde{E}^A_i&:= \{ x \in V\setminus W' : |N^+_D(x)\cap W'_i| < k \}, \\
\widetilde{E}^A_0 &:= \{x\in V\setminus W': |N^+_D(x)\setminus W'| < \frac{\delta_0^+}{4} \},  \\
\widetilde{F}^A &:= \{ x \in V\setminus W': \big|\{ i'\in [t]: x\in \widetilde{E}^A_{i'}\}\big| > \frac{t}{50m^{1/2}k} \}.
\end{split}
\end{align}

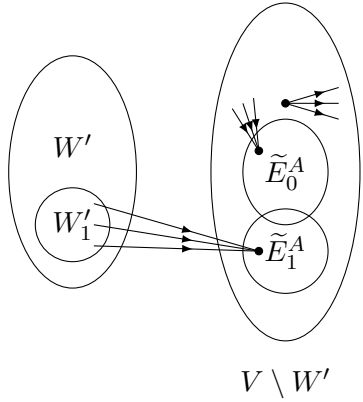
\begin{figure}
\centering
\begin{tikzpicture}[scale=0.7]

  \draw[fill=none] (0,0) ellipse [x radius=1.2,y radius=2.2];
  \node at (0,0.5) {$W'$};

  \draw[fill=none] (0,-1) ellipse [x radius=0.7,y radius=0.7];
  \node at (0,-1) {$W'_1$};

  \draw[fill=none] (4,0) ellipse [x radius=1.4,y radius=3.2];
  \node at (4,-4) {$V\setminus W'$};

  \draw[fill=none] (4,-1.5) ellipse [x radius=0.8,y radius=0.8];
    \draw[fill=none] (4,0) ellipse [x radius=0.8,y radius=1];
  \node at (4,0) {$\widetilde{E}^A_0$};
  
    \node at (4,-1.5) {$\widetilde{E}^A_1$};

\filldraw[fill=black] (3.5,-1.5) circle (2pt);

\draw[middlearrow={latex}] (0.4,-1) -- (3.5,-1.5) ;
\draw[middlearrow={latex}] (0.4,-0.6) -- (3.5,-1.5) ;
\draw[middlearrow={latex}] (0.4,-1.4) -- (3.5,-1.5) ;

\filldraw[fill=black] (3.5,0.4) circle (2pt);


\draw[middlearrow={latex}] (3,1.2) -- (3.5,0.4) ;
\draw[middlearrow={latex}] (3.2 ,1.3) -- (3.5,0.4);
\draw[middlearrow={latex}] (3.4, 1.4) -- (3.5,0.4);

\filldraw[fill=black] (4,1.3) circle (2pt);

\draw[middleuparrow={latex}] (4,1.3) -- (5,1) ;
\draw[middleuparrow={latex}] (4,1.3) -- (5,1.3) ;
\draw[middleuparrow={latex}] (4,1.3) -- (5,1.6) ;

\end{tikzpicture}
\caption{$\widetilde{E}^A_0$, $\widetilde{E}^A_1$ and in/out-neighbours guaranteed at a vertex in each set.}\label{fig: EA}
\end{figure}

For a vertex $v\in \widetilde{E}_i^A$, it has at most $k-1$ out-neighbors in $W'_i$ and thus by~\ref{W'1}, there are at least $k$ in-neighbors of $v$ in $W'_i$. 
We can find many in-neighbours in $V\setminus W'$ of a vertex in $\widetilde{E}_0^A$ and we can find many out-neighbors in $V\setminus W'$ of a vertex in $(V\setminus W')\setminus \widetilde{E}_0^A$. See Figure~\ref{fig: EA}.


One of the issues is that we do not wish to assign too many vertices to a fixed $W'_i$ so that we can guarantee~\ref{A1}. In order for this, we will first define a set $I_A(x)$ of the indices $i\in [t]$ such that $x$ can be assigned to $W_i$ provided that some set $N^A(x)$ of additional vertices are assigned together with $x$. We will then choose $i(x) \in I_A(x)$ in the way that, for each $i\in [t]$, not too many vertices $x$ satisfies $i(x)=i$. We will further choose a set $N^A(x)$ of $k$ vertices and assign vertices in $N^A(x)\cup \{x\}$ together to $W_{i(x)}$.

As the definition of $E^A(i,j)$ with \eqref{eq: E def} implies 
$\widetilde{E}^A_i \subseteq E^A_i$, we have
\begin{align}
\begin{split}\label{eq: wide tilde F size}
|\widetilde{F}^A| &\leq \frac{50m^{1/2}k}{t}\sum_{i \in [t]}|\widetilde{E}^A_i|\stackrel{\eqref{eq: E size}}{\leq} \frac{50 m^{1/2} k t \delta_0^+}{10^{6}k^2(k+\ell)^2  m^2 t } \leq
\frac{\delta^+_0}{2 \cdot 10^4 k^3 m^{3/2} }.
\end{split}
\end{align}
For each $x\in \widetilde{F}^A$, we define
\begin{align}\label{eq: I(x) def}
I_A(x):= \left\{\begin{array}{ll}
\{i\in [t] : |N^-_D(x)\cap W'_i|\geq k\} & \text{if } x\in \widetilde{F}^A\setminus \widetilde{E}^A_0, \\
\{i\in [t] :|N^+_D(x)\cap W'_i|\geq k\} & \text{if } x\in \widetilde{F}^A\cap \widetilde{E}^A_0.
\end{array}\right.
\end{align}

Note that for each $i\in [t]$, if $x\in \widetilde{E}^A_i$, then \ref{W'1} with \eqref{eq: tilde E def} implies that $| N^-_D(x)\cap W'_i| \geq k$.
Thus, for each $x\in  \widetilde{F}^A\setminus \widetilde{E}^A_0$, the set $I_A(x)$ contains all $i \in [t]$ with $x \in \widetilde{E}^A_i$, so we have
\begin{align}\label{eq: IA size 1}
|I_A(x)| \geq \big|\{i'\in [t]: x\in \widetilde{E}_{i'}^A\}\big| \stackrel{\eqref{eq: tilde E def}}{\geq } \frac{t}{50 m^{1/2} k}.
\end{align}
For each $x\in \widetilde{F}^A\cap \widetilde{E}^A_0$, we have 
$$\delta_0^+ \leq d_D^+(x)\leq |N_D^+(x)\setminus W'| + \sum_{i\in [t]\setminus I_A(x) } k + \sum_{i\in I_A(x)} |W'_i| \stackrel{\eqref{eq: tilde E def},\ref{W'1}}{\leq }\delta_0^+/4 + kt + \frac{|I_A(x)|n}{180mt}.$$
Thus by \eqref{eq: n min deg}, for each $x\in \widetilde{F}^A\cap \widetilde{E}^A_0$ we have
\begin{align}\label{eq: IA size 2}
|I_A(x)| \geq \frac{100 \delta_0^+ mt}{n}.
\end{align}
Thus for each $x\in \widetilde{F}^A$, by \eqref{eq: IA size 1} and \eqref{eq: IA size 2}, we can choose $i(x) \in I_A(x)$ in such a way that the following holds for each $i\in [t]$.
\begin{align}
\begin{split}\label{eq: index choice}
\big|\{ x\in  \widetilde{F}^A :i(x)=i \}\big| &\leq \max \left (\Big\lceil \frac{ |\widetilde{F}^A| }{t/(50m^{1/2} k)}\Big\rceil, \Big\lceil\frac{|\widetilde{F}^A|}{100 \delta_0^+mt/n } \Big\rceil\right )
\stackrel{ \eqref{eq: wide tilde F size}}{\leq } \frac{n}{300 k^2 m t}.
\end{split}
\end{align}
For each $x\in \widetilde{F}^A\setminus \widetilde{E}^A_0$, we have
\begin{align}\label{eq: N(x)1}
|N^+_D(x)\setminus (W' \cup E^A_{i(x)}\cup E^B_{i(x)})| 
\stackrel{\eqref{eq: tilde E def}}{\geq} \delta^+_0/4 - |E^A_{i(x)}\cup E^B_{i(x)}| \stackrel{\eqref{eq: E size},\eqref{eq: wide tilde F size}}{\geq} k|\widetilde{F}^A|.
\end{align}
Similarly, for every $x\in \widetilde{F}^A\cap \widetilde{E}^A_0$, we have
\begin{eqnarray}\label{eq: N(x)2}
|N^-_D(x)\setminus (W' \cup E^A_{i(x)}\cup E^B_{i(x)})|
&\geq& |V\setminus W'| - |N^+_D(x)\setminus W'| - (n- d_D(x)) - |E^A_{i(x)}\cup E^B_{i(x)}| \nonumber \\
&\stackrel{\ref{W'1},\eqref{eq: E size}}{\geq }& \left (1- \frac{1}{180m} \right )n - \delta^+_0/4 - \ell - \delta^+_0/4  \stackrel{\eqref{eq: wide tilde F size}}{\geq} k|\widetilde{F}^A|,
\end{eqnarray}
where the last inequality follows since we have $\ell < n/10$ by~\eqref{eq: n min deg}.
Hence, for every $x\in \widetilde{F}^A$, we can choose a set $N^A(x)$ of the size $k$
such that $N^A(x)\cap N^A(y) =\emptyset$ for all $x\neq y \in \widetilde{F}^A$ and
\begin{align}\label{eq: N(x) def}
N^A(x) \subseteq \left\{\begin{array}{ll} 
N^+_D(x)\setminus (W' \cup E^A_{i(x)}\cup E^B_{i(x)}) &\text{if } x\in \widetilde{F}^A\setminus \widetilde{E}^A_0, \\
N^-_D(x)\setminus (W' \cup E^A_{i(x)}\cup E^B_{i(x)}) & \text{if } x\in \widetilde{F}^A\cap \widetilde{E}^A_0.
\end{array}\right.
\end{align}
For each $i\in [t]$, let 
$$W^{*}_i:= W'_i\cup \bigcup_{x\in \widetilde{F}^A \colon i(x)=i} (N^A(x)\cup \{x\}) \enspace \text{and} \enspace W^{*}:= \bigcup_{i\in [t]} W^{*}_i.$$

\begin{CLAIM}\label{cl: W* exists}
The sets $W^*_1,\dots, W^*_t$ satisfy the following.
\begin{enumerate}[label=\text{\rm (W$^*$\arabic*)}]
\item \label{W*0} $W^*_1,\dots, W^*_t$  are pairwise disjoint.
\item\label{W*1} $W'_i\subseteq W^*_i$ and $\widetilde{F}^A\subseteq W^*$.
\item \label{W*2} $3k +\ell \leq |W^*_i| \leq \frac{n}{80mt}$.
\item \label{W*3} For each $v\in W^*_i$, both $(v, V'\setminus(C\cup E^A_i))$ and  $(V'\setminus (C \cup E^B_i),v )$ are $k$-connected in $D[W^*_i]$.
\end{enumerate}
\end{CLAIM}
\begin{proof}
As $i(x)$ is well-defined for all $x\in \widetilde{F}^{A}$ and $W'_1,\dots, W'_t$ are pairwise disjoint, the definition of $W^*_i$ implies \ref{W*0}. By the definition, \ref{W*1} is clear. For every $i\in [t]$, we have
$$|W^*_i| \leq |W'_i| +  (k+1)\big|\{x\in \widetilde{F}^A: i(x)=i\}| \stackrel{\ref{W'1},\eqref{eq: index choice}}{\leq}  \frac{n}{180mt} + \frac{n }{150 k m t} \leq \frac{n}{80mt}.$$
Together with the fact that $|W^*_i|\geq |W'_i| \stackrel{\ref{W'1}}{\geq} 3k+\ell$, the condition~\ref{W*2} follows. 
To prove \ref{W*3}, by \ref{W'2}, we only have to show that for all $i\in [t]$ and $v\in W^*_i\setminus W'_i$, both $(v, V'\setminus(C\cup E^A_i))$ and  $(V'\setminus (C\cup E^B_i),v )$ are $k$-connected in $D[W^*_i]$. 
If $v\in N^A(x)$ for some $x\in \widetilde{F}^A$ with $i(x)=i$, then $v \in V'\setminus (C\cup E^A_i\cup E^B_i)$, thus we are done.
If $v\in \widetilde{F}^A$ with $i(v)=i$, then 
\eqref{eq: I(x) def} together with \eqref{eq: N(x) def} implies that 
$v$ has at least $k$ out-neighbors and $k$ in-neighbors in $W'_i\cup N^A(v)$. 
Thus \ref{W*3} holds.
\end{proof}

\noindent {\bf Step 4.2. Distribution of vertices in $\widetilde{F}^B$.}
Now we will define the set $\widetilde{F}^B$ of vertices which are exceptional with respect to the partition $W^*_1,\dots, W^*_{t}$. Since $\widetilde{F}^A \subseteq W^*$,
all vertices in $V\setminus W^*$ has at least $k$ out-neighbors in $W'_i\subseteq W^*_i$ for many indices $i\in [t]$, which will simplify our analysis.

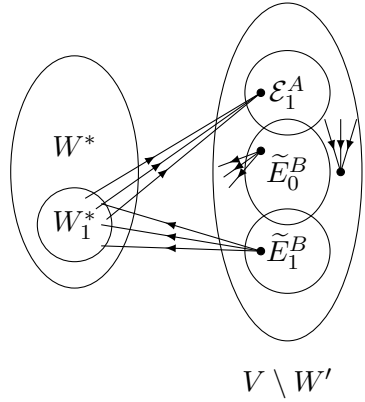
\begin{figure}
\centering
\begin{tikzpicture}[scale=0.7]

  \draw[fill=none] (0,0) ellipse [x radius=1.2,y radius=2.2];
  \node at (0,0.5) {$W^*$};

  \draw[fill=none] (0,-1) ellipse [x radius=0.7,y radius=0.7];
  \node at (0,-1) {$W^*_1$};

  \draw[fill=none] (4,0) ellipse [x radius=1.4,y radius=3.2];
  \node at (4,-4) {$V\setminus W'$};

  \draw[fill=none] (4,-1.5) ellipse [x radius=0.8,y radius=0.8];
    \draw[fill=none] (4,0) ellipse [x radius=0.8,y radius=1];
  \node at (4,0) {$\widetilde{E}^B_0$};
  
    \node at (4,-1.5) {$\widetilde{E}^B_1$};
    
      \draw[fill=none] (4,1.5) ellipse [x radius=0.8,y radius=0.8];
    \node at (4,1.5) {$\mathcal{E}^A_1$};

\filldraw[fill=black] (3.5,-1.5) circle (2pt);

\draw[middlearrow={latex}] (3.5,-1.5) -- (0.5,-1) ;
\draw[middlearrow={latex}] (3.5,-1.5) -- (0.5,-0.6) ;
\draw[middlearrow={latex}] (3.5,-1.5) -- (0.5,-1.4);

\filldraw[fill=black] (3.5,0.4) circle (2pt);


\draw[middleuparrow={latex}]  (3.5,0.4) -- (2.7,0.1);
\draw[middleuparrow={latex}]  (3.5,0.4) -- (2.8 ,-0.1);
\draw[middleuparrow={latex}]  (3.5,0.4) -- (2.9, -0.3);

\filldraw[fill=black] (5,0) circle (2pt);

\draw[middlearrow={latex}]  (4.7,1) -- (5,0);
\draw[middlearrow={latex}] (5,1) -- (5,0);
\draw[middlearrow={latex}] (5.3,1)  -- (5,0);

\filldraw[fill=black] (3.5,1.5) circle (2pt);

\draw[middledownarrow={latex}] (0.2,-0.5)  -- (3.5,1.5);
\draw[middledownarrow={latex}]  (0.4,-0.7) -- (3.5,1.5);
\draw[middledownarrow={latex}] (0.6,-0.9) -- (3.5,1.5);

\end{tikzpicture}
\caption{$\mathcal{E}^A_1$, $\widetilde{E}^B_0$, $\widetilde{E}^B_1$ and in/out-neighbours guaranteed at a vertex in each set.}\label{fig: EB}
\end{figure}

Again, for each $x\in \widetilde{F}^B$, we aim to find $i(x)\in [t]$ and a vertex-set $N^B(x)$ so that we can assign $x$ together with $N^B(x)$ to $W_{i(x)}$.
We define the following ``updated'' sets of exceptional vertices. For each $i\in [t]$,
\begin{align}\begin{split}\label{eq: B E def}
\mathcal{E}^A_i&:= \{x\in V\setminus W^{*}: |N^+_D(x)\cap W^{*}_i| <k\}, \\
\widetilde{E}^B_i&:= \{ x \in V\setminus W^{*} : |N^-_D(x)\cap W^{*}_i| < k \},\\
\widetilde{E}^B_0& := \{x\in V\setminus W^{*}: |N^-_D(x)\setminus W^{*}| < \delta_0^-/4 \}, \text{ and} \\
\widetilde{F}^B &:= \{ x \in V\setminus W^{*}: \big|\{ i'\in [t]: x\in \widetilde{E}^B_{i'}\}\big| > \frac{t}{100m^{1/2}k} \}. 
\end{split}
\end{align}
For a vertex $v\in \mathcal{E}_i^A$, it has at most $k-1$ out-neighbors in $W^*_i$ and thus by~\ref{W'1}, there are at least $k$ in-neighbors of $v$ in $W^*_i$. Similarly a vertex in $\widetilde{E}^B_i$ has at least $k$ out-neighbours in $W^*_i$.
Moreover, we can find many out-neighbours in $V\setminus W'$ of a vertex in $\widetilde{E}_0^B$ and we can find many in-neighbors in $V\setminus W'$ of a vertex in $(V\setminus W')\setminus \widetilde{E}_0^B$. See Figure~\ref{fig: EB}.

As $W'_i\subseteq W^*_i$ for all $i\in [t]$, \eqref{eq: tilde E def} implies $\mathcal{E}^A_i \subseteq \widetilde{E}^A_i$.
The definition of $E^B(i,j)$ with \eqref{eq: E def} implies 
$\widetilde{E}^B_i \subseteq E^B_i$. Hence
\begin{align}
\begin{split}\label{eq: wide tilde F size B}
|\widetilde{F}^B| &\leq \frac{100m^{1/2}k}{t}\sum_{i \in [t]}|\widetilde{E}^B_i|\stackrel{\eqref{eq: E size}}{\leq} \frac{100  m^{1/2}k t \delta_0^-}{10^{6}k^2(k+\ell)^2  m^2 t } \leq
\frac{\delta^-_0}{10^4 k^3 m^{3/2} }.
\end{split}
\end{align}

For  $x\in \widetilde{F}^B\setminus \widetilde{E}^B_0$, the vertex $x$ has at least $\delta_0^-/4$ in-neighbors in $V\setminus W^*$. As we did in Step~4.1, we wish to choose $k$ non-exceptional in-neighbors of $x$ in $V\setminus W^*$. 
However, $\delta_0^-/4$ could be much smaller than $|\mathcal{E}^A_i|$, as $d_D^-(v) \ll \delta_0^+$ could happen while $|\mathcal{E}^A_i|$ is proportional to $\delta_0^+$.
Hence, we may not be able to choose an appropriate set $N^B(x)\subseteq N^{-}_D(x)$ if we determine $i(x)$ first. To resolve this issue, for $x\in \widetilde{F}^B\setminus \widetilde{E}^B_0$, we will choose $N^B(x)$ first and then choose $i(x)$ later, while for $x\in \widetilde{F}^B\cap\widetilde{E}^B_0$, we will choose $i(x)$ first and then choose $N^B(x)$ later.

For each $x\in \widetilde{F}^B\setminus \widetilde{E}^B_0$, we have
$$
|N^{-}_D(x) \setminus (W^{*}\cup \widetilde{F}^B)| 
\geq \delta^-_0/4 - |\widetilde{F}^B| \stackrel{\eqref{eq: wide tilde F size B}}{\geq} k |\widetilde{F}^B|.
$$
Thus, we can choose sets $N^B(x) \subseteq N^{-}_D(x) \setminus (W^{*}\cup \widetilde{F}^B)$ with $|N^B(x)|=k$ in such a way that $N^B(x)\cap N^B(y)=\emptyset$ for all $x\neq y\in \widetilde{F}^B\setminus \widetilde{E}^B_0$.
For each $x\in \widetilde{F}^B\setminus \widetilde{E}^B_0$, let 
\begin{align}\label{eq: IBx def}
I_B(x):=\big\{ i\in [t]: |N^+_D(x)\cap W_i^*| \geq k\big\} \cap \bigcap_{y\in N^B(x)}\big\{ i\in [t]: y\notin  \widetilde{E}^A_i\cup \widetilde{E}^B_i \big\}.
\end{align}
This definition ensures that we can assign $\{x\}\cup N^B(x)$ to $W^*_i$ as long as $i\in I_B(x)$.
Note that for all $x\in \widetilde{F}^B$ and $z\in N^B(x)$, we have $z \notin \widetilde{F}^A$ (as $\widetilde{F}^A \subseteq W^*$), and $z\notin \widetilde{F}^B$.
Thus \eqref{eq: tilde E def} and \eqref{eq: B E def} imply that, for each $y\in N^B(x)$, we have
$$|\{ i\in [t]: y \notin \widetilde{E}^A_i\cup \widetilde{E}^B_i \}| \geq t - \frac{t}{50m^{1/2}k} - \frac{t}{100m^{1/2}k} \geq \left (1 - \frac{1}{30m^{1/2}k} \right )t.$$
This with \eqref{eq: tilde E def} implies that for every $x\in \widetilde{F}^B\setminus \widetilde{E}^B_0$,
\begin{align}\label{eq: IB size 1}
|I_B(x)| \geq |\{ i\in [t]: x\notin \widetilde{E}^A_i \}| - \frac{kt}{30 m^{1/2} k}
\geq (1- \frac{1}{50m^{1/2}k})t  - \frac{k t}{30 m^{1/2} k} \geq t/2,
\end{align}
where the first inequality follows since we have $\mathcal{E}^A_i \subseteq \widetilde{E}^A_i$ for each $i\in [t]$, and the second inequality follows since we have $x\notin \widetilde{F}^A \subseteq W^*$. Let 
$$N^B_1:=\bigcup_{x\in \widetilde{F}^B\setminus \widetilde{E}^B_0}N^B(x).$$
Then 
\begin{align}\label{eq: NB1 size}
|N^B_1| = \sum_{x\in \widetilde{F}^B\setminus \widetilde{E}^B_0} |N^B(x)| = k |\widetilde{F}^B|.
\end{align}

Now we wish to choose sets $I_B(x)$ and $N^B(x)$ for the remaining vertices $x \in \widetilde{F}^B\cap \widetilde{E}^B_0$ as follows. For every $x\in \widetilde{F}^B\cap \widetilde{E}^B_0$, we define 
\begin{align}\label{eq: I(x) def B}
I_B(x):=
\{i\in [t] : |N^-_D(x)\cap W^*_i| \geq k\}.
\end{align}
For each $x\in \widetilde{F}^{B}\cap \widetilde{E}^B_0$, we have
$$\delta_0^- \leq d_D^-(x)\leq |N_D^-(x)\setminus W^{*}| + \sum_{i\in [t]\setminus I_B(x) } k + \sum_{i\in I_B(x)} |W^{*}_i| \stackrel{\ref{W*2}}{\leq} \delta_0^-/4 + kt + \frac{|I_B(x)|n}{80mt}.$$ 
Thus by \eqref{eq: n min deg}, we have 
\begin{align*}
|I_B(x)|  \geq \frac{50 \delta_0^- mt}{n}.
\end{align*}
Now we choose $i(x) \in I_B(x)$ for all $x\in \widetilde{F}^B$, regardless whether it is in $\widetilde{F}^B\cap \widetilde{E}^B_0$ or in $\widetilde{F}^B\setminus \widetilde{E}^B_0$.
Together with \eqref{eq: IB size 1}, for every $x\in \widetilde{F}^B$, we can choose $i(x)\in I_B(x)$ such that for every $i\in [t]$,
\begin{align}\label{eq: covered by i}
|\{x\in \widetilde{F}^B : i(x) =i \}| \leq \max\left( \Big\lceil \frac{ |\widetilde{F}^B|}{t/2}\Big\rceil, \Big\lceil \frac{|\widetilde{F}^B|}{ 50\delta_0^-mt/n }\Big\rceil \right)\stackrel{\eqref{eq: wide tilde F size B}}{\leq} \frac{n}{10^3 k^3 mt}.
\end{align}
For every $x \in  \widetilde{F}^{B}\cap \widetilde{E}^B_0$, we have
\begin{eqnarray}\label{eq: N(x)B 2}
|N^+_D(x)\setminus (W^{*}\cup E^A_{i(x)} \cup E^B_{i(x)})|
&\geq & |V\setminus W^{*}| - |N^-_D(x)\setminus W^*| - (n- d_D(x)) - |E^A_{i(x)}\cup E^B_{i(x)}| \nonumber \\
&\stackrel{\ref{W*2},\eqref{eq: E size}}{\geq }& \left (1- \frac{1}{80m} \right )n - \delta^-_0/4 - \ell - \frac{n}{10^{6}} \nonumber \\
&\geq& \frac{n}{3} \stackrel{\eqref{eq: wide tilde F size B}}{\geq} 2k|\widetilde{F}^B| .
\end{eqnarray}
Thus by \eqref{eq: NB1 size},  for each $x\in \widetilde{F}^B\cap \widetilde{E}^B_0$, 
we can choose sets $N^B(x)\subseteq N^+_D(x)\setminus (W^{*}\cup E^A_{i(x)} \cup E^B_{i(x)} \cup N^B_1)$ in such a way that 
$N^B(x)\cap N^B(y) =\emptyset$ for all $x\neq y \in  \widetilde{F}^B\cap \widetilde{E}^B_0$. Hence $\{N^B(x) : x\in \widetilde{F}^B\}$ forms a collection of pairwise disjoint subsets of size $k$. As $i(x)\in I_B(x)$, together with \eqref{eq: IBx def} it follows that
\begin{align}\label{eq: N(x) def B}
N^B(x) \subseteq \left\{\begin{array}{ll} 
N^-_D(x)\setminus (W^*\cup E^A_{i(x)}\cup E^B_{i(x)}) &\text{if } x\in \widetilde{F}^B\setminus E^B_0, \\
N^+_D(x)\setminus (W^*\cup E^A_{i(x)}\cup E^B_{i(x)}) & \text{if } x\in \widetilde{F}^B\cap E^B_0.
\end{array}\right.
\end{align}
For each $i\in [t]$, let 
$$W_i:= W^{*}_i\cup \bigcup_{ x\in \widetilde{F}^B : i(x) = i} \big(\{x\}\cup N^B(x)\big).$$

\begin{CLAIM}\label{cl: W exist}
The following hold.
\begin{enumerate}[label=\text{\rm (W\arabic*)}]
\item \label{W0} $W_1,\dots, W_t$ are pairwise disjoint.
\item \label{W1} $W^*_i\subseteq W_i$ and $|W_i| \leq \frac{n}{50mt}$.
\item \label{W2} $\widetilde{F}^A\cup \widetilde{F}^B\subseteq \bigcup_{i\in [t]} W_i$.
\item \label{W3} For each $v\in W_i$, both $(v, V'\setminus(C\cup E^A_i))$ and 
$(V'\setminus (C\cup E^B_i),v )$ are $k$-connected in $D[W_i]$.
\end{enumerate}
\end{CLAIM}
\begin{proof}
Note that \ref{W0} is obvious from \ref{W*0} and the definition of $N^B(x)$.
Note that \ref{W2} is obvious as $i(x)$ is defined for all $x\in \widetilde{F}^A\cup \widetilde{F}^B$.
For each $i\in [t]$, 
$$|W_i| \leq |W^*_i| + (k+1)\big|\{ x\in \widetilde{F}^B: i(x)=i\}\big| \stackrel{\ref{W*2},\eqref{eq: covered by i}}{\leq } \frac{n}{80mt} + 
\frac{n (k+1)}{10^3 k^3 mt}\leq \frac{n}{50mt}.$$
Thus we have \ref{W1}.
To show \ref{W3}, by Lemma~\ref{lem:glue} together with~\ref{W*3}, it suffices to show that for every $v\in W_i\setminus W^*_i$, both $(v,W^*_i)$ and $(W^*_i,v)$ are $k$-connected in $D[W_i]$.
If $v\in N^B(x)$ for some $x\in \widetilde{F}^B$ with $i(x)=i$, then $v \in V'\setminus (C\cup E^A_i\cup E^B_i)$ by~\eqref{eq: N(x) def B}, and thus we are done.
If $v\in \widetilde{F}^B$ with $i(v)=i$, then 
\eqref{eq: I(x) def B}, \eqref{eq: IBx def} and \eqref{eq: N(x) def B} altogether imply that 
$v$ has at least $k$ out-neighbors and $k$ in-neighbors in $W^*_i\cup N^B(v)$. This proves \ref{W3} and completes the proof of the claim.
\end{proof}
\vspace{0.3cm}

\noindent {\bf Step 5. Finishing the proof of Lemma~\ref{lem:main}} 
To finish the proof of  Lemma~\ref{lem:main}, we now verify \ref{A5}--\ref{A4}. By \ref{W0}, $W_1,\dots, W_t$ are pairwise disjoint.
Note that \ref{W1} implies \ref{A1}, \ref{W2} with \eqref{eq: tilde E def} and \eqref{eq: B E def} implies \ref{A3}, and~\ref{W'0} implies~\ref{A5}.

For each $i\in [t]$, because $C \subseteq \bigcup_{i'=1}^{t}W_{i'}$ and $C_i\subseteq W_i$, \eqref{eq: E def} imply that every vertex  $v\in V\setminus (E^A_i\cup E^B_i\cup \bigcup_{i'\in [t]} W_{i'})$ satisfies $|N^+_D(x)\cap W_i|,|N^-_D(x)\cap W_i| \geq k$. We also have
$$\Big|V\setminus \big(E^A_i\cup E^B_i\cup \bigcup_{i'\in [t]} W_{i'} \big) \Big|
\stackrel{\eqref{eq: E size}}{\geq} \Big|V\setminus  \bigcup_{i'\in [t]} W_{i'} \Big| - \frac{n}{10^6 k^2 (k+\ell)^2 m^2} \stackrel{\ref{W1}}{\geq } (1- \frac{1}{10^5 k^4m^2})\Big|V\setminus \bigcup_{i\in [t]} W_{i'}\Big|,$$
hence \ref{A4} follows.

Now it only remains to show~\ref{A2}. Let $u,v \in V(W_i)$ and $S\subseteq V(W_i)$ with $|S| \leq k-1$. By \ref{W3}, both $(u,V'\setminus(C\cup E^A_i))$ and $(V'\setminus (C\cup E_i^B),v)$ are $k$-connected in $D[W_i]$.
Thus there exist a path $P^A$ from $u$ to a vertex $u'\in V'\setminus(C\cup E^A_i)$ in $D[W_i]\setminus S$ and a path $P^B$ from a vertex $v'\in V'\setminus(C\cup E^B_i)$ to $v$ in $D[W_i]\setminus S$.
By~\eqref{eq: E def}, there are $J^A, J^B \subseteq [3k]$ with $|J^A|,|J^B|\geq 2k$ 
 such that $u' \notin E^A(i,j)$ for each $j\in J^A$ and $v'\notin E^B(i,j')$ for each $j'\in J^B$.
 Note that $|J^A\cap J^B|\geq k > |S|$. 
As $3k$ sets $\left \{ A_{i,s}\cup B_{i,s}\cup V(P_{i,s}) \right \}_{s \in [3k]}$ are pairwise disjoint by~\ref{AB1} and~Proposition~\ref{prop: W'}, there exists $j\in J^A\cap J^B$ such that $S\cap (A_{i,j} \cup B_{i,j} \cup V(P_{i,j}) ) =\emptyset$.
Since $j\in J^A\cap J^B$, there exist $u''\in N^+_D(u')\cap A_{i,j}$ and $v''\in N^-_D(v')\cap B_{i,j}$. By \ref{AB2}, there exist a path $Q^A$ in $D[A_{i,j}]$ from $u''$ to $a_{i,j}$ and a path $Q^B$ in $D[B_{i,j}]$ from $b_{i,j}$ to $v''$.
Then $P^A$, $u'u''$, $Q^A$, $P_{i,j}$, $Q^B$, $v''v'$ and $P^B$ together form a directed walk from $u$ to $v$ in $D[W_i]\setminus S$. Therefore, there exists a path from $u$ to $v$ in $D[W_i]\setminus S$ and we obtain \ref{A2}. This completes the proof of Lemma~\ref{lem:main}.

\section{Derivation of main results}\label{sec:proof_balanced}
In this section, we prove Theorem~\ref{thm:cycle} and Theorem~\ref{thm:balanced}. Both theorems will be derived from Lemma~\ref{lem:main} by using appropriate auxiliary bipartite graphs.

\begin{proof}[Proof of Theorem~\ref{thm:balanced}]
Let $D$ be a strongly $10^8 qk^2\ell(k+\ell)^2tm^2\log{m}$-connected $n$-vertex digraph with $\delta(D)\geq n-\ell$, and $Q_1 , \dots , Q_t \subseteq V(D)$ be $t$ disjoint sets with $|Q_i| \leq q$. For each $i\in [t]$, we have $a_i\in \mathbb{N}$ with $\sum_{i\in [t]} a_i \leq n$ and $a_i\geq n/(10tm)$. Let us define $b_1 := a_1 + \left (n - \sum_{i=1}^{t}a_i \right )$ and $b_i := a_i$ for $2 \leq i \leq t$. Then $\sum_{i=1}^{t}b_i = n$.

As $D$ is a strongly $10^8 qk^2\ell(k+\ell)^2tm^2\log{m}$-connected $n$-vertex digraph with $\delta(D)\geq n-\ell$, there are pairwise disjoint sets $W_1,\dots, W_t$ satisfying \ref{A1}--\ref{A5} by Lemma~\ref{lem:main}.
For every $i\in [t]$, it follows that
\begin{align} \label{eq: size}
a_i -|W_i|  \stackrel{\ref{A1}}{\geq} \frac{n}{10tm}-\frac{n}{50tm} \geq \frac{n}{20tm}.
\end{align}
Let 
$$A:= V(D)\setminus \bigcup_{i\in [t]} W_i.$$
For each $i\in [t]$, let 
$$U_i:=\big\{ x\in A: |N_D^+(x)\cap W_i|\geq k  \text{ and } |N_D^-(x)\cap W_i| \geq k\big\}.$$
Then for each $i\in [t]$, we have
\begin{align}\label{eq: Ui size A4}
|U_i| \stackrel{\ref{A4}}{\geq} (1- \frac{1}{10^5k^4 m^2})|A|. 
\end{align}
 For each $i \in [t]$ and $j \in [b_i - |W_i|]$, we define a new vertex $v_{i,j}$ and let $B$ be the set of all these new vertices. Let $H$ be an auxiliary bipartite graph with vertex partition $(A,B)$ such that
$$E(H):= \bigcup_{i\in [t]} \big\{ x v_{i,j}:  j\in [b_i-|W_i|], x\in U_i  \big\}.$$
Note that we have
\begin{align}\label{eq: A B size}
(1-\frac{1}{50m})n \stackrel{\ref{A1}}{\leq} |A| = n - \sum_{i\in [t]} |W_i| = \sum_{i\in [t]} (b_i- |W_i| ) = |B| \leq n.
\end{align}
Moreover, for each $x\in A$, 
\ref{A3} with the fact $b_i\geq a_i$ implies that 
$$d_H(x) = \sum_{i\in [t]: x\in U_i }(b_i-|W_i|) 
\stackrel{\eqref{eq: size}}{\geq } \frac{n}{20tm} |\{ i\in [t]: x\in U_i\}| 
\stackrel{\ref{A3}}{\geq} \frac{n}{20tm} \left (1- \frac{1}{30k} \right )t  \geq \frac{n}{30m}.$$
For every $i\in [t]$ and $v_{i,j} \in B$, we have
$$d_H(v_{i,j}) = |U_i| \stackrel{\eqref{eq: Ui size A4},\eqref{eq: A B size}}{\geq}
 (1- \frac{1}{40m})n.$$
Thus for all $x\in A$, $i\in [t]$ and $v_{i,j}\in B$, we have
$$d_H(x)+ d_{H}(v_{i,j})\geq  (1- \frac{1}{40m})n + \frac{n}{30m} > |A|.$$
This together with Fact~\ref{fact: perfect matching} implies that 
$H$ has a perfect matching $M$. For each $i\in [t]$, let 
$$V_i:= \big\{ x\in A : x v_{i,j} \in M \text{ for some } j\in [b_i-|W_i|]\big\}.$$
Then by the definition of $H$, we have $V_i\subseteq U_i$ and 
$V_1,\dots, V_t$ forms a partition of $A$ such that 
$|V_i| = b_i - |W_i|$ for each $i\in [t]$. Now we remove $b_1 - a_1\stackrel{\eqref{eq: size}}{ \leq}  |V_1|$ arbitrary vertices from $V_1$, then for each $i\in [t]$ we have $|V_i| = a_i - |W_i|$.

By the definition of $U_i$, each vertex $v\in V_i$ has at least $k$ in-neighbors and at least $k$ out-neighbors in $W_i$.
Hence \ref{A2} together with Lemma~\ref{lem:glue} implies that 
$D[W_i\cup V_i]$ is strongly $k$-connected.
As $\left \{W_i\cup V_i\right \}_{i \in [t]}$ forms a partition of $V(D)$ and $|W_i\cup V_i| = |W_i|+ a_i-|W_i| = a_i$ for every $i \in [t]$ and $Q_i \subseteq W_i$ by~\ref{A5}, this gives a desired partition.
\end{proof}

The proof of Theorem~\ref{thm:cycle} is more involved than the proof of Theorem~\ref{thm:balanced}. In addition to the auxiliary matching techniques used before, we also use Corollary~\ref{cor:twocycles} in the proof.

\begin{proof}[Proof of Theorem~\ref{thm:cycle}]
Let $T$ be a strongly $10^{50} t$-connected $n$-vertex tournament with any $t$ distinct vertices $x_1 , \dots , x_t$.
For each $i\in [t]$, let $\ell_i \in \mathbb{N}$ with $\ell_i\geq 3$ and $\sum_{i\in [t]} \ell_i = n$. 

The idea is as follows. Let $\ell_1\geq \dots \geq \ell_t$.
We will use Lemma~\ref{lem:main} to partition $V(T)$ into $t$ sets $W_1,\dots, W_t$ satisfying \ref{A1}--\ref{A5}. By using \ref{A1}--\ref{A5}, we will distribute vertices outside $\bigcup_{i\in [t]}W_i$ to each $W_i$ to obtain partition $W_1\cup W^*_1\cup V_1,\dots, W_t\cup W^*\cup V_t$ of $V(T)$ such that each $T[W_i\cup W^*_i\cup V_i]$ is strongly $10^9$-connected with size $\ell'_i$ such that the following holds for some $s\in [t]$. 
$$\ell_i \geq \ell'_i \text{ for each } i\in [s] \text{ and } \ell'_i \geq \ell_i \text{ for each } i\in [t]\setminus [s].$$ Furthermore, we will ensure that there exists a partition $J_1,\dots, J_s$ of $[t]\setminus [s]$ such that 
$\ell_i - \ell'_i = \sum_{j\in J_i} (\ell'_j-\ell_j)$ for all $i\in [s]$.
By applying Corollary~\ref{cor:twocycles} to $T[W_j\cup W^*_j\cup V_j]$ for each $j\in J_i$ and $i\in [s]$, we will partition it into two cycles of size $\ell_j$ and $\ell'_j-\ell_j$.  
By using collected properties, we will be able to combine all the cycles of length $\ell'_j-\ell_j$ with $j\in J_i$ together with $W_i\cup W^*_i\cup V_i$, we will obtain the desired collections of cycles of prescribed lengths.

By permuting indices if necessary, we assume that 
$\ell_1\geq \dots \geq \ell_t$.
Let 
\begin{align}\label{eq: s def}
s := \max\left\{ i\in [t]: \ell_i >\frac{n}{30t}\right\}.
\end{align}
Note that we have $s \geq 1$ as $\sum_{i\in [t]}\ell_i = n$.

We apply Lemma~\ref{lem:main} to $T$ with $10^9,1,10,1$ and $t$ playing the roles of $k,\ell,m,q$ and $t$, respectively to obtain pairwise disjoint sets $W_1,\dots, W_t$ satisfying the following for every $i\in [t]$, where $W:= \bigcup_{i\in [t]} W_i$ and
$U_{i}:=\{ v\in V(T)\setminus W: |N_D^+(v)\cap W_i|\geq 10^9 \text{ and } |N_D^-(v)\cap W_i| \geq 10^9 \}$.
\begin{enumerate}[label=\text{\rm (C\arabic*)}]
\item \label{C1} $|W_i| \leq \frac{n}{500t}$ and $n':= |V(T)\setminus W| \geq (1- \frac{1}{500})n$.
\item \label{C2} $T[W_i]$ is strongly $10^9$-connected.
\item \label{C3}$\left|\big\{i'\in [t] :|N_T^{+}(w)\cap W_{i'}|\geq 10^9 \text{ and }
|N_T^{-}(w)\cap W_{i'}| \geq 10^9 \big\}\right|\geq (1- \frac{1}{9 \cdot 10^{10}})t$ holds for every $w \in V(T) \setminus W$.
\item \label{C4}$|U_i| \geq (1- \frac{1}{10^{43}})|V(T)\setminus W|$.
\item \label{C5} $x_i \in W_i$.
\end{enumerate}
Now we take a partition of $[t]\setminus [s]$ satisfying the following properties.
\begin{CLAIM}\label{cl: partition [t][s]}
There exists a partition $J_1,\dots, J_s$ of $[t]\setminus [s]$ satisfying the following.
\begin{enumerate}[label=\text{\rm (J\arabic*)}]
\item\label{J1} For every $i\in [s]$, we have $|W_{i}| + |J_{i}|\cdot  \lceil \frac{n}{5t} \rceil + \frac{n}{40t} \leq \ell_{i}$.
\end{enumerate}
\end{CLAIM}
\begin{proof}
We may assume that $s<t$, otherwise the claim is obvious. Let $J_1,\dots, J_s$ be a collection of sets with maximum $|\bigcup_{i\in [s]} J_i|$ among all collections of pairwise disjoint subsets of $[t]\setminus [s]$ satisfying~\ref{J1}.
As $J_1=\dots =J_s =\emptyset$ satisfies~\ref{J1}, such a choice exists. For some $j \in [t]\setminus [s]$, if $j \notin \bigcup_{i\in [s]} J_{i}$ then for every $i\in [s]$, we have
$$|W_{i}| + (|J_{i}|+1)  \lceil \frac{n}{5t} \rceil + \frac{n}{40t}  > \ell_{i},$$
otherwise $J_1,\dots, J_{i-1}, J_{i}\cup \{j\}, J_{i+1},\dots, J_{s}$ still satisfies \ref{J1}, contradicting the maximality of $J_1,\dots, J_s$.
However, we have
\begin{eqnarray*}
n &=& \sum_{i\in [t]} \ell_i 
< \sum_{i\in [s]} \Big(|W_i| + (|J_i|+1) (\frac{n}{5t}+1) + \frac{n}{40t}\Big) + \sum_{i\in [t]\setminus[s]} \ell_i  \\
&\stackrel{\eqref{eq: s def}, \ref{C1}}{<}& \frac{sn}{500t} + \frac{(t-s)n}{5t} + (t-s) + \frac{sn}{5t}+ s + \frac{sn}{40t} + \frac{(t-s)n}{30t} < n,
\end{eqnarray*}
where we obtain the inequalities as $0<s<t$. Hence the choice of $J_1,\dots, J_s$ gives a partition of $[t]\setminus [s]$ satisfying \ref{J1}.
\end{proof}
For all $i\in [s]$ and $j\in J_i$, 
$$|U_i\cap U_j | \stackrel{\ref{C4}}{\geq} (1-\frac{2}{10^{43}})|V(T)\setminus W| \stackrel{\ref{C1}}{>} \lceil \frac{n}{20t} \rceil (t-s).$$
Thus, for all $i\in [s]$ and $j\in J_i$, we can choose a set
\begin{align}\label{eq: W*sizees}
W^*_j \subseteq U_i\cap U_j  \text{ with }  |W^*_j| = \lceil \frac{n}{20t} \rceil -|W_j|
\end{align}
in such a way that $W^*_{s+1},\dots, W^*_{t}$ are pairwise disjoint, and let $W^*_i := \emptyset$ for $i \in [s]$.

For every $i\in [t]$, let us define
\begin{align}\label{eq: bi sizess}
b_i:= \left\{\begin{array}{ll}
\ell_i - |W_i| - |J_i|\cdot  \lceil \frac{n}{5t} \rceil\stackrel{\ref{J1}}{\geq}\frac{n}{40t} & \text{if } i\in [s] \\
\ell_i + \lceil \frac{n}{5t} \rceil -\lceil \frac{n}{20t} \rceil \geq \frac{3n}{20t} & \text{if } i\in [t]\setminus[s] .
\end{array}\right.
\end{align}
Then we have 
\begin{eqnarray}\label{eq: bi sum}
\sum_{i\in [s]} (b_i +|W_i| )+ \sum_{i\in [t]\setminus[s]} (b_i+ |W_i\cup W^*_i|)
\stackrel{\eqref{eq: W*sizees}}{=} \sum_{i\in [t]} \ell_i -  \sum_{i\in [s]} |J_i| \lceil \frac{n}{5t} \rceil + \sum_{i\in [t]\setminus [s]} \lceil \frac{n}{5t} \rceil = \sum_{i\in [t]}\ell_i  =n,
\end{eqnarray}
since $\left \{ J_1,\dots, J_s \right \}$ is a partition of $[t]\setminus[s]$.
Now we are ready to define an auxiliary bipartite graph $H$.
We consider the vertex set $A$ and a set $B$ of 
new vertices as follows.
$$A:= V(T) \setminus (W\cup \bigcup_{i\in [t]} W^*_i) \enspace \text{and}\enspace
B:=\{ v_{i,j} : i\in [t], j\in [b_i] \}.$$
Moreover, by \eqref{eq: bi sum}, we have $|A|=|B|$. 
Then as $|W_i|+|W_i^*| = \lceil\frac{n}{20t}\rceil$ holds for each $i\in [t]\setminus[s]$ by \eqref{eq: W*sizees}, \ref{C1} implies
\begin{align}\label{eq: A B size2}
\frac{9n}{10}\leq |A|=|B|\leq n.
\end{align}
Let $H$ be a bipartite graph with vertex partition $(A,B)$ and
$$E(H):= \bigcup_{i\in [t]} \{ x v_{i,j}:  j\in [b_i], x\in U_i  \} .$$
Then 
\ref{C3} implies that, for each $x\in A$ 
$$d_H(x) = \sum_{i\in [t]\colon x\in U_i } b_i 
\stackrel{\eqref{eq: bi sizess}}{\geq } \frac{n}{40t} \big|\{ i\in [t]: x\in U_i\}\big| 
\geq \frac{n}{40t} \left (1- \frac{1}{9 \cdot 10^{10}} \right )t  \geq \frac{n}{50}.$$
Also, for all $i\in [t]$ and $v_{i,j} \in B$, we have
\begin{eqnarray*}
d_H(v_{i,j}) = |U_i \setminus \bigcup_{i'\in [t]} W^*_{i'}|
\stackrel{\ref{C4}}{\geq} |V(T)\setminus (W\cup \bigcup_{i'\in [t]} W^*_{i'})|
- \frac{1}{10^{43}}|V(T)\setminus W| 
  \stackrel{\eqref{eq: A B size2},\ref{C1}}{\geq} \left (1 -  \frac{2}{10^{43}} \right )|A|.
\end{eqnarray*}
Hence for every $x\in A$, $i\in [t]$ and $v_{i,j}\in B$, it follows that 
$$d_H(x)+ d_{H}(v_{i,j})\geq  \frac{n}{50} +  (1 -  \frac{2}{10^{43}}  ) |A| \stackrel{\eqref{eq: A B size2}}{>}|A|.$$ 
This together with Fact~\ref{fact: perfect matching} implies that $H$ has a perfect matching $M$. For every $i\in [t]$, let 
$$V_i:= \{ x\in A : x v_{i,j} \in M \text{ for some } j\in [b_i]\}.$$
Then by the definition of $H$, we have $V_i\subseteq U_i$ and 
$V_1,\dots, V_t$ form a partition of $A$ such that 
$|V_i| = b_i$ for each $i\in [t]$.
Also $W^*_i \subseteq U_i$ for all $i\in [t]$.
By the definition of $U_i$, each vertex $v\in W^*_i \cup V_i$ has at least $10^9$ in-neighbors and at least $10^9$ out-neighbors in $W_i$. By~\ref{C2} together with Lemma~\ref{lem:glue}, $T[W_i\cup W^*_i\cup V_i]$ is strongly $10^9$-connected.
Moreover, for all $i\in [s]$ and $j\in[t]\setminus [s]$, we have
\begin{align}\label{eq:WW*V size}
|W_i\cup W^*_i\cup V_i| \stackrel{\eqref{eq: bi sizess}}{=} \ell_i - |J_i|\lceil \frac{n}{5t}\rceil \enspace\text{and}\enspace 
|W_j\cup W^*_j\cup V_j| \stackrel{\eqref{eq: bi sizess}}{=} \ell_j + \lceil \frac{n}{5t} \rceil.
\end{align}
Note that we have $\ell_i\geq 3$ and $\lceil \frac{n}{5t}\rceil\geq 3$ for each $i\in [t]\setminus [s]$. Thus by Corollary~\ref{cor:twocycles}, for each $i\in [t]\setminus[s]$, there exist two vertex-disjoint cycles $C_i$ and $C'_i$ in $T[W_i\cup W^*_i\cup V_i]$ with $x_i \in V(C_i)$, $V(C_i)\cup V(C'_i) = W_i\cup W^*_i\cup V_i$ and $|C_i| = \ell_i$ and $|C'_i| = \lceil \frac{n}{5t}\rceil$.

By~\eqref{eq: s def},~\ref{C1} and~\eqref{eq: W*sizees}, for every $j\in [t]\setminus [s]$, we have $|W^*_j| \geq \lceil \frac{n}{20t} \rceil - \frac{n}{500t} > \ell_j$. Thus for each $j\in [t]\setminus [s]$, there is a vertex $y_j \in V(C'_j)\cap W^*_j$.
By~\eqref{eq: W*sizees}, for all $i\in[s]$ and $j\in J_i$, it follows that $y_j \in U_i$.
For each $i\in [s]$, we let 
$$\widetilde{V}_i:= W_i\cup V_i \cup \bigcup_{j\in J_i} V(C'_j).$$
As $V_i\subseteq U_i$, each $x\in V_i$ has at least $10^9$ in-neighbors and out-neighbors in $W_i$.
For all $j\in J_i$ and $x\in V(C'_j)$, 
there exists a path from $x$ to $y_j$ on $C'_j$, and a vertex $y_j\in U_i$ has at least one out-neighbor in $W_i$, together we obtain a path from $x$ to $W_i$ in $T[\widetilde{V}_i]$. Similarly, we can also obtain a path from $W_i$ to $x$ in $T[\widetilde{V}_i]$. By~\ref{C2} and Lemma~\ref{lem:glue}, we conclude that $T[\widetilde{V}_i]$ induces a strongly connected tournament.
By Camion's theorem (Theorem~\ref{thm:camion}), for each $i\in [s]$, the tournament $T[\widetilde{V}_i]$ contains a cycle $C_i$ of length
$$|\widetilde{V}_i| = |W_i|+ b_i + |J_i| \lceil \frac{n}{5t}\rceil 
\stackrel{\eqref{eq: bi sizess}}{=} \ell_i.$$

Since $\big\{\widetilde{V}_i \big\}_{i \in [s]} \cup \left \{V(C_i) \right \}_{i \in [t]\setminus [s]}$ is a partition of $V(T)$, $C_1,\dots, C_t$ form a set of vertex-disjoint $t$ cycles with $x_i \in V(C_i)$ and $|V(C_i)| = \ell_i$ for $i \in [t]$. This completes the proof.
\end{proof}

\bibliographystyle{abbrv}
\bibliography{references-ecd}

\end{document}